\RequirePackage{rotating}
\documentclass[a4paper, english, 9pt]{amsart}

\usepackage{listings}
\usepackage{rotating}

\usepackage{amsthm,amssymb,url,hyperref}
\usepackage{fullpage}
\usepackage{MnSymbol}
\usepackage[all]{xy}
\usepackage{tikz}
\usepackage{subfig} 
\usepackage{caption}
\usepackage{forest}
\usepackage{color}
\usepackage{amsaddr}
\usepackage{mathtools}
\usepackage{enumitem}
\usepackage{hyperref}
\usepackage{booktabs} 

\theoremstyle{plain}
\newtheorem{theorem}{Theorem}[section]

\newtheorem*{theorem*}{Theorem}

\newtheorem{corollary}[theorem]{Corollary}

\newtheorem{lemma}[theorem]{Lemma}

\newtheorem{proposition}[theorem]{Proposition}

\theoremstyle{definition}

\def\ker#1{\mathrm{ker}(#1)}

\def\aut#1{\mathrm{Aut}(#1)}

\def\End#1{\mathrm{End}(#1)}
\def\aff#1{\mathrm{Aff}#1}
\def\Aff#1{\mathrm{Aff}#1}
\def\lmlt{\mathrm{LMlt}}
\def\dis{\mathrm{Dis}}
\def\fix{\mathrm{Fix}}
\def\sym{\mathrm{Sym}}
\def\syl{\mathrm{Syl}}
\def\Sym{\mathrm{Sym}}

\def\con{\mathrm{Con}}

\def\comment#1{{\color{red} #1}}
\def\c#1{\mathrm{con}_{#1}}
\def\Q{\mathcal{Q}_{\mathrm{Hom}}}

\def\setof#1#2{\{#1\,\colon\,#2\}}

\def\Z{\mathbb Z}

\def\Ma{\mathcal Max}

\def\ma{\mathcal Min}

\def\Q{\mathcal Q}

\def\LSS{\mathcal{LSS}}

\def\cg#1{\equiv_\alpha}
\newcommand*\xbar[1]{%
   \hbox{%
     \vbox{%
       \hrule height 0.5pt 
       \kern0.5ex
       \hbox{%
         \kern-0.1em
         \ensuremath{#1}%
         \kern-0.1em
       }%
     }%
   }%
} 
\def\fs#1{{\color{cyan} #1}}

\title{latin quandles of size $16p$}
\author{Marco Bonatto \and Filippo Spaggiari}
\email[1]{marco.bonatto.87@gmail.com}
\email[2]{spaggiari@karlin.mff.cuni.cz}

\begin{document}

\begin{abstract}
    In this paper we obtain the classification of latin quandles of size $16p$ where $p$ is an odd prime. In particular we have that such quandles are always subdirectly reducible but in some cases for small primes. Specifically, we have that latin quandles of size $16p$ are affine if $p\neq 1 \pmod{3}$ or $p\neq 3,5$. If $p=1\pmod 3$ there are $2$ subdirectly reducible non directly decomposable latin quandles of size $16p$ for every prime with $p=1\pmod{3}$ and there are one subdirectly irreducible latin quandle of size $16p$ for $p=3,5$. We provide explicit constructions as coset quandles for all the quandles mentioned above.
\end{abstract}
\maketitle

\section*{Introduction}

Quandles are binary algebras introduced in the framework of knot theory to construct knot invariants \cite{J, Matveev}. Beyond their origins, they have proven to be deeply connected to various other mathematical structures, such as solutions to the set-theoretical quantum Yang-Baxter equation \cite{etingof2001indecomposable, etingof1999set} and pointed Hopf algebras \cite{AG}. These connections underline the importance of quandles as versatile and foundational objects in both algebra and topology.

Quandles have been extensively studied through both group and module theory. In addition, universal algebra offers an important framework for investigating the complex relationship between properties of the congruence lattice and those of the so-called displacement group, which encodes many properties of the quandle itself. Indeed, there are two Galois connections between the congruence lattice of a quandle and the lattice of normal subgroups of the left multiplication group contained in the displacement group. This rich interplay between the two lattices has proven crucial in addressing classification problems for finite quandles, facilitating a deeper understanding of their structure and properties. Moreover, it has also been fundamental in developing commutator theory for quandles. Commutator theory, initially developed for general algebraic structures in \cite{comm}, has been specialized for quandles in \cite{CP}.

In this paper, we use this set of tools to obtain the classification of latin quandles of size $16p$, where $p$ is a prime, taking advantage of the classification of latin quandles of size $4,16,p$, and $4p$ already available \cite{ESS, LSS, RIG}, as well as the homogeneous representation of connected quandles over the displacement group \cite{J}. Using this machinery, we can translate the classification problem into a purely group-theoretical one. Universal algebraic techniques are primarily used to identify the groups of interest for obtaining a homogeneous representation, assuming that the congruence lattice of the structure is known.

The core of the paper lies in Section \ref{sec:latin16p}, where we also collect existing results on the classification of small latin quandles and prove the main classification results for latin quandles of size $16p$ (completed in Section \ref{computational} with some computational results concerning the cases $p=3,5$). This classification, achieved through a detailed case analysis, demonstrates the power of the aforementioned techniques. Specifically, the classification strategy proceeds as follows:
\begin{itemize}
\item[(i)] Determine the structure of the congruence lattice;
\item[(ii)] Use the Galois correspondence between the congruence lattice and the lattice of normal subgroups of the left multiplication group to identify the displacement group;
\item[(iii)] Establish whether the groups satisfying the required properties exist and, if they do, whether their automorphisms give rise to coset quandles with the desired properties (in this case, being latin of size $16p$).
\end{itemize}
In the paper, we take advantage of the fact that finite latin quandles are solvable (in the sense of \cite{comm}) and so the description of finite simple latin quandles is completely understood. In particular, such quandles are affine over elementary abelian $p$-groups. 

The algorithm for classification is used to investigate directly indecomposable quandles. On the other hand, directly decomposable latin quandles of size $16p$ are known since the quandles of size $4,16,p$ and $4p$ are classified. If $p\neq 1\pmod{3}$ the directly decomposable latin quandles of size $16p$ are affine, as they are the direct product of the $9$ latin quandles of size $16$ and one of the $p-2$ latin quandles of size $p$ and such quandles are affine. If $p= 1\pmod{3}$ we also have the direct product of the unique latin quandle of size $4$ and one of the two non-affine latin quandles of size $4p$ described in \cite{LSS}.

Combining the computational results from Section \ref{computational}, Proposition \ref{group K}, and Theorem \ref{sr quandles} (in which we also provide an explicit coset representation for the quandles obtained) with the enumeration mentioned above, we achieve our main classification results, summarized in Table \eqref{enum}.


\begin{table}[ht]
  \centering
  \begin{tabular}{cccc} 
    \toprule
    $p$ & Subdirectly irreducible & Directly decomposable (DD) & Subdirectly reducible (not DD) \\
    \midrule
    $p=1\pmod{3}$ & 0 & $9(p-2)+2$ & $2$ \\
    $p\neq 1\pmod{3}, \, p\neq 3,5 $ & 0 & $9(p-2)$ & 0 \\
    3 & 1 & 9 & 0 \\
    5 & 1 & 27 & 0 \\
    \bottomrule
  \end{tabular}\caption{Enumeration of latin quandles of size $16p$.}\label{enum}
\end{table}

%

In detail, the structure of the paper is as follows. Section \ref{sec:preliminary} collects foundational concepts from group theory and universal algebra. These preliminaries establish the necessary background for subsequent sections and include results concerning the congruence lattice of specific structures. For additional information on universal algebra, the reader may refer to \cite{bergman2011universal}, which provides an accessible introduction to the subject.

Section \ref{sec:quandles} introduces essential notations and concepts from quandle theory, including related groups of automorphisms. This section also presents key theorems addressing isomorphism problems and connections between the normal subgroups of the displacement group and the congruences of the quandle, laying the groundwork for the main results of the paper.

Section \ref{sec:latin16p} contains all the classification results, including the computational ones.

Finally, Section \ref{sec:appendix} presents the purely group-theoretical results used throughout the paper. 

\section{Preliminary results} \label{sec:preliminary}
\subsection{Group Theory}

We denote by $\Sym_Q$ the symmetric group over a non-empty set $Q$. An \emph{action} of a group $G$ on a set $Q$ is a group homomorphism $$\rho:G\longrightarrow \Sym_Q,\quad g\mapsto \rho_g.$$ We say that an action is \emph{faithful} if the map $\rho$ is injective, i.e. $\ker{\rho}=1$. For $x\in Q$, we denote by $G_x=\{g\in G\,\colon\, \rho_g(x)=x\}$ the \emph{stabilizer} of $x$ in $G$, and by $x^G=\{\rho_g(x)\,\colon\ g\in G\}$ the \emph{orbit} of $x$ with respect to the action of $G$.

We denote by $\aut{G}$ the group of automorphisms of a group $G$. The \emph{inner automorphism} with respect to $g\in G$ is the automorphism $\widehat{g}$ defined by setting 
\[
\widehat{g}(x) = gxg^{-1}, \quad \text{for every } x \in G.
\]
 
 If $h\in \aut{G}$ and $N$ is a $h$-invariant normal subgroup of $G$ (i.e. $h(N)=N$) then $h$ induces an automorphism on the factor $G/N$.

\begin{lemma}
    Let $N$ be a normal subgroup of a group $G$, and let $f\in\aut{G}$ be such that $f(N)=N$. Then the map $$f_N:G/N\longrightarrow G/N,\quad xN \mapsto f(x)N$$ is a well defined automorphism of $G/N$.
\end{lemma}
\begin{proof}
    The map $xN\mapsto f(x)N$ is well defined. In fact, if $x^{-1}y\in N$ then $f(x^{-1}y)=f(x)^{-1}f(y)\in N$. Clearly $f$ is onto. Moreover 
    \begin{align*}
    f_N((xN)^{-1} yN)&=f_N(x^{-1}yN)=f(x^{-1}y)N=f(x^{-1})f(y)N\\&=f(x)^{-1}N f(y)N=f_N(xN)^{-1}f_N(yN).        
    \end{align*}
 Therefore $f_N\in \End{G/N}$. If $f_N(xN)=f(x)N=N$ then $f(x)\in N$ and $x\in N$ accordingly. So $f_N$ is injective.
\end{proof}

Let $N$ and $H$ be two groups, and let $\rho: H \to \aut{N}$ be a group action. The \emph{semidirect product} of $N$ and $H$ induced by $\rho$, denoted $N \rtimes_\rho H$, is the group $(N \times H,\cdot)$ where the operation $\cdot$ is defined by setting
\[
(n_1, h_1) \cdot (n_2, h_2) = (n_1 \rho_{h_1}(n_2), h_1 h_2).
\]
for every $n_1,n_2\in N$ and $h_1,h_2\in H$.

\subsection{Universal Algebra}

An {\it operation} of finite arity $n$ on a set $A$ is a map from $A^n$ to $A$. A (finitary) \emph{algebra} is a pair $\langle A, F\rangle$ where $A$ is a non-empty set and $F=\left\langle f_i: i \in I\right\rangle$ is a family of operations on $A$, indexed by some set $I$. The operations $\setof{f_i}{i\in I}$ are the  {\it basic operations} of $(A,F)$. The function $\rho: I \rightarrow \mathbb{N}$ which assigns to each $i \in I$ the arity of $f_i$ is called \emph{similarity type} of A. 
A {\it term} is any meaningful composition of basic operations in $F$. A {\it Malt'sev term} is a ternary term $m(x,y,z)$ satisfying $m(y,x,x)\approx m(x,x,y)\approx y$. An algebra with a Malt'sev term is called a {\it Malt'sev algebra}.




A \emph{subalgebra of $A$} is a subset $B\subseteq A$ that is closed with respect to every basic operation in $F$.

A \emph{congruence} on an algebra $(A,F)$ is an equivalence relation $\theta$ on $A$ such that for every basic operation $f \in F$ of arity $n$ and every $a_i, b_i \in A$ with $a_i\, \theta\, b_i$ for $i = 1, \ldots, n$, we have $f(a_1, \ldots, a_n)\,\theta\, f(b_1, \ldots, b_n)$. The congruence lattice of an algebra $(A,F)$ is denoted by $\con(A,F)$ (or simply by $\con(A)$). We denote the congruence class of an element $x$ with respect to $\alpha\in \con(A)$ by $[x]_\alpha$ (or simply by $[x]$) and the set of classes by $A/\alpha$. The top element of $\con(A)$ is $1_A=A\times A$ and the bottom element of $\con(A)$ is the identity relation, denoted by $0_A$.

An algebra is said to be:

\begin{enumerate}
    \item \emph{directly decomposable} if it can be expressed as a direct product of two non-trivial algebras. Equivalently, an algebra is directly decomposable if and only if there are two congruences $\alpha$ and $\beta$ such that $\alpha\wedge\beta=0_A$ and $\alpha\circ\beta=1_A$ ($\circ$ is the usual composition of binary relations). If an algebra is not directly decomposable, it is called \emph{directly indecomposable};

\item \emph{subdirectly reducible} if there exists $\setof{\alpha_i}{i\in I}\subseteq \con(A)\setminus \{0_A\}$ such that $\wedge_{i\in I} \alpha_i= 0_A$. Otherwise, $A$ is said to be {\it subdirectly irreducible}. It is known that subdirectly irreducible algebras are exactly those algebra having a smallest non-trivial congruence. Subdirectly irreducible algebras are clearly directly indecomposable;

\item \emph{simple} if it has no non-trivial proper congruences, i.e., the only congruences of $A$ are $0_A$ and $1_A$. 
\end{enumerate}

A \emph{homomorphism} between two algebras $(A, F)$ and $(B, G)$ of the same similarity type is a function $\varphi: A \rightarrow B$ that preserves all operations in $F$. Specifically, for each $n$-ary basic operation $f_i \in F$, the function $\varphi$ satisfies $\varphi(f_i(a_1, \ldots, a_n)) = g_i(\varphi(a_1), \ldots, \varphi(a_n))$ for all $a_1, \ldots, a_n \in A$. If $\varphi$ is a bijection, then $\varphi$ is called an \emph{isomorphism}.

An \emph{automorphism} of an algebra $(A, F)$ is an isomorphism from the algebra to itself. The set of all automorphisms of an algebra with the composition of maps is a group called \emph{automorphism group} of the algebra. We denote the group of automorphisms of $(A,F)$ by $\aut{A,F}$ (or simply by $\aut{A}$).

The commutator theory for general algebras was introduced in \cite{comm} in order to extend the well-known concept of commutator of subgroups to the notion of commutator of congruences of any algebraic structure. The notion of {\it central} and {\it abelian} congruences is needed in the broader framework of general algebra to define {\it solvability} and {\it nilpotence} through the existence of special series of congruences named \emph{central series} and \emph{derived series} as for groups. To do so, given an algebra $A$ and $\alpha,\beta$ being congruences of $A$, the {\it commutator} $[\alpha,\beta]$ is a congruence of $A$ satisfying a certain term condition that for groups reduces to the usual notions of the commutator of normal subgroups. We are not going to explicitly describe such conditions and we refer the reader to \cite{comm} for further details (and to \cite{CP} for the case of quandles). Based on the notion of commutator of congruences, we say that $\alpha$ is {\it abelian} (resp. {\it central}) if $[\alpha,\alpha]=0_A$ (resp. $[\alpha,1_A]=0_A$). The largest central congruence is called the {\it center} of $A$ and is denoted by $\zeta_A$. 

The derived and central series are defined, respectively, as 
\begin{align*}
    &\gamma^0(A)=\gamma_0(A)=1_A,\quad \gamma^{i+1}(A)=[\gamma^i(A),\gamma^i(A)], \quad \gamma_{i+1}(A)=[\gamma_i(A),1_A],\text{ for } i\geq 0.
\end{align*}

An algebra is {\it solvable} (resp. {\it nilpotent}) of length $n$ if $\gamma^n(A)=0_A$ (resp. $\gamma_n(A)=0_A$). If $1_A$ is abelian, then $A$ is said to be {\it abelian} (equivalently $\gamma_1(A)=0_A$ or $\zeta_A=1_A$). The congruence $\gamma_1(A)$, simply denoted by $\gamma_A$, is the smallest congruence of $A$ with the abelian factor. Thus, given a congruence $\alpha$, the factor algebra $A/\alpha$ is abelian if and only if $\gamma_A \leq \alpha$. In particular, solvable simple algebras are abelian (indeed, if $A$ is simple and solvable then $\gamma_A\neq 1_A$ and so necessarily $\gamma_A=0_A$).

\subsection{Congruence slim algebras}\label{slim}

Let $A$ be an algebra. Recall that a {\it maximal} (resp. {\it minimal}) congruence is a congruence $\alpha\neq 0_A,1_A$ such that if $\alpha< \beta$ (resp. $\beta<\alpha$) then $\beta=1_A$ (resp. $\beta=0_A$). We denote by $\Ma_A$ (resp. $\ma_A$) the set of the maximal (resp. minimal) congruences of $A$ and we define 
\begin{align*}
\mu_A =\bigwedge_{\alpha\in \Ma_A} \alpha, \qquad
\nu_A =\bigvee_{\alpha\in \ma_A} \alpha.
\end{align*}

\begin{lemma}\label{cong of directly irred}
Let $A$ be a finite, directly indecomposable Malt'sev algebra. Then: 
\begin{enumerate}
    \item[(i)] $\nu_A\leq \mu_A$. 
    \item[(ii)] $\Ma_A\cap \ma_A\neq \emptyset$ if and only if $\con(A)$ is the three-element chain.
\end{enumerate}

\end{lemma}

\begin{proof}
(i) Let $\alpha\in \ma_A$ and $\beta\in \Ma_A$. Recall that for Malt'sev algebra we have $\alpha\circ \beta=\alpha\vee\beta$. Then $\alpha\wedge\beta\leq \alpha$ and $\beta\leq\alpha\circ \beta$. Hence, if $\alpha$ is not contained in $\beta$ we have that $\alpha\wedge\beta=0_A$ and $\alpha\circ\beta=1_A$ implying that $A$ is directly decomposable. Thus $\alpha\leq \beta$ for all $\alpha\in \ma_A$ and $\beta\in \Ma_A$, then accordingly $\nu_A\leq\mu_A$. 

(ii) Clearly, if $\con(A)$ is the chain of length $3$, then the unique proper congruence is both maximal and minimal. On the other hand, assume that $\alpha\in \Ma_A\cap \ma_A$. Then we have that
$$\nu_A\leq \mu_A\leq \alpha\leq \nu_A\leq \mu_A$$
and therefore $\nu_A= \mu_A=\alpha$, i.e. $\Ma_A=\ma_A=\{\alpha\}$. If $\beta$ is a congruence of $A$ then it is contained in a maximal congruence, and so $\beta\leq \alpha$. Hence, either $\beta=0_A$ or $\beta=\alpha$. So we see that $\alpha$ is the unique proper non-trivial congruence of $A$. 
%
%
\end{proof}

Let $A$ be an algebra $A$. The \emph{height} of $\con(A)$ is the maximum of the chains' lengths of $\con(A)$. If $\con(A)=\Ma_A\cup \ma_A\cup \{0_A,1_A\}$ we say that $A$ is \emph{congruence slim}. We have that congruence slim algebras are those for which the height of $\con(A)$ is less than or equal to $4$.

\begin{lemma}\label{llength 4}
Let $A$ be a directly indecomposable algebra. The following are equivalent: 
\begin{itemize}
    \item[(i)] $A$ is congruence slim. 
    \item[(ii)] $\con(A)$ has height less than or equal to $4$. 
\end{itemize}    
\end{lemma}

\begin{proof}

(i) $\Rightarrow$ (ii) Let $0_A< \alpha_1< \ldots \alpha_n< 1_A$ be a chain of distinct congruences of maximal length. Then $\alpha_1\in \ma_A, \alpha_{2}\in \Ma_A$, so $n=2$ and then the maximal length of a chain of congruences is $4$.

(ii) $\Rightarrow$ (i) Let $\alpha$ be a congruence that is not in $\Ma_A\cup \ma_A\cup \{0_A,1_A\}$. Then there exists $\beta,\delta\in \con(A)$ such that $0_A<\beta<\alpha<\delta<1_A$. Then we have a chain of length $5$, which is a contradiction. 
\end{proof}

\begin{proposition}\label{lattices}
Let $A$ be a directly indecomposable, congruence slim Malt'sev algebra. Then one of the following holds:
\begin{itemize}
\item[(i)] $\con(A)$ is a chain of length less than or equal to $4$.
\item[(ii)] $\con(A)$ is
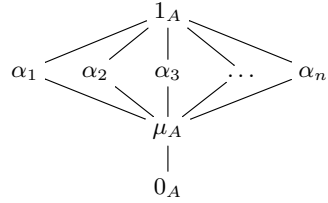
\begin{center}
		$\xymatrixrowsep{0.15in}
		\xymatrixcolsep{0.15in}
		\xymatrix{ 
			& & 1_A  \ar@{-}[d]\ar@{-}[dr]\ar@{-}[drr] \ar@{-}[dl] \ar@{-}[dll] & &\\
			\alpha_1 \ar@{-}[drr] & \alpha_2\ar@{-}[dr] &\alpha_3\ar@{-}[d] & \ldots\ar@{-}[dl] & \alpha_n \ar@{-}[dll]\\
			& & \mu_A\ar@{-}[d] &  &\\
			& & 0_A & &
		}$
		\end{center}
		for some $n\in \mathbb{N}$.
		
		\item[(iii)] $\con(A)$ is
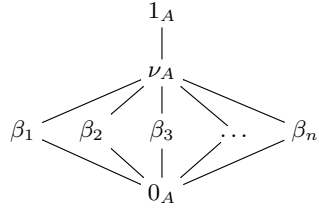
\begin{center}
		$\xymatrixrowsep{0.15in}
		\xymatrixcolsep{0.15in}
		\xymatrix{ 
	& & 1_A\ar@{-}[d] & &\\		
			& & \nu_A  \ar@{-}[d]\ar@{-}[dr]\ar@{-}[drr] \ar@{-}[dl] \ar@{-}[dll] & &\\
			\beta_1 \ar@{-}[drr] & \beta_2\ar@{-}[dr] &\beta_3\ar@{-}[d] & \ldots\ar@{-}[dl] & \beta_n \ar@{-}[dll]\\
			& & 0_A &  &\\
		}$
		\end{center}
		for some $n\in \mathbb{N}$.
\end{itemize}
\end{proposition}

\begin{proof}
Recall that for Malt'sev algebras it holds that $\alpha\circ\beta=\alpha\wedge\beta$ for every pair of congruences $\alpha,\beta$. Let us denote $L=\con(A)$. According to Lemma \ref{llength 4}, we have $L=\Ma_A\cup \ma_A\cup \{0_A,1_A\}$. 
\begin{itemize}
    \item  If $|\Ma_A|=|\ma_A|=1$ then $L$ is a chain of length less than or equal to $4$. 

\item Assume that $|\Ma_A|>1$ and let $\alpha,\beta\in\Ma_A$ with $\alpha\neq \beta$. Note that $\alpha\circ \beta=1_A$. If $\alpha\wedge\beta=0_A$ then $A$ is directly decomposable, thus $0_A\neq \alpha\wedge\beta$ and, moreover, $\alpha\wedge\beta\in \ma_A$. Therefore, $\mu_A\leq \alpha\wedge\beta \leq \nu_A\leq \mu_A$, where the last inequality follows from Lemma \ref{cong of directly irred}(i). Hence $\mu_A=\nu_A\in \ma_A$. Thus, $A$ has a unique minimal congruence, namely $\mu_A$ and it follows that $L$ is as in (ii).

\item Assume that $|\ma_A|>1$ and let $\alpha,\beta\in \ma_A$ with $\alpha\neq \beta$. A dual argument shows $\alpha\vee\beta=\nu_A=\mu_A\in \Ma_A$ for every pair of distinct minimal congruences $\alpha,\beta$. So, $A$ has a unique maximal congruence, namely $\nu_A$ and so $L$ is as in (iii).\qedhere
\end{itemize}

%
%
\end{proof}

For solvable algebras, we can obtain further information about the congruence lattice.
\begin{lemma}\label{gamma and max} Let $A$ be a solvable algebra. Then $\gamma_A\leq \mu_A$.
\end{lemma}

\begin{proof}
Since $A$ is solvable, the simple factors of $A$ are abelian, therefore $A/\alpha$ is abelian for all $\alpha\in\Ma_A$. So,  $\gamma_A\leq\alpha$ for every $\alpha\in \Ma_A$ and thus $\gamma_A\leq \mu_A$. 
\end{proof}

\begin{corollary}\label{mu and gamma}
Let $A$ be a finite directly indecomposable solvable non-abelian Malt'sev algebra. If $A$ is congruence slim and $|\Ma_A|> 1$ then 
$\gamma_A	=\mu_A$ 
\end{corollary}

\begin{proof}
By Lemma \ref{gamma and max} and Proposition \ref{lattices}(ii), we have $0_A\neq \gamma_A\leq \mu_A\in \ma_A$. Therefore, equality holds.
\end{proof}

\section{Quandles} \label{sec:quandles}

\subsection{Left quasigroups and quandles}
A {\it left quasigroup} is a set $Q$ endowed with a pair of binary operations $F=\{*,\backslash\}$, such that 
\begin{align}\label{LQG}
x*(x\backslash y)=y=x\backslash(x*y)
\end{align}
hold for every $x,y\in Q$. We usually denote $*$ just by juxtaposition. Given a left quasigroup $Q$ we can define the left and right multiplication maps as
$$L_x:y\mapsto x*y,\quad R_x:y\mapsto y*x$$
for every $x\in Q$. According to axioms \eqref{LQG}, we see that the map $L_x$ is a bijection of $Q$ for every $x\in Q$. So we have that $x\backslash y=L_x^{-1}(y)$ for every $x,y\in Q$ and the {\it left multiplication group} of $Q$ as $\lmlt(Q)=\langle L_x\colon\, x\in Q\rangle$.

The {\it Cayley kernel} of $Q$ is the equivalence relation $\lambda_Q$ defined by setting $x\, \lambda_Q\, y$ if and only if $L_x=L_y$. We say that $Q$ is:
\begin{enumerate}[label=(\roman*)]
    \item {\it faithful} if $\lambda_Q=0_Q$;
    \item {\it connected} if the natural action of $\lmlt(Q)$ is transitive on $Q$;
    \item {\it latin} if $R_x$ is bijective for every $x\in Q$.
\end{enumerate}
Note that latin left quasigroups are faithful and connected.

A {\it quandle} $Q$ is a left quasigroup such that the following axioms hold:
\begin{enumerate}[label=(\roman*)]
    \item $x*x=x$ for all $x\in Q$;
    \item $x*(y*z)=(x*y)*(x*z)$ for all $x,y,z\in Q$.
\end{enumerate}

Note that (ii) is equivalent to having $L_x\in \aut{Q}$ for every $x\in Q$ and accordingly $\lmlt(Q)\leq \aut{Q}$.
Quandles can be constructed using groups and automorphisms as follows: let $G$ be a group, $f\in \aut{G}$ and $H\leq Fix(f)=\setof{x\in G}{f(x)=x}$. The set $G/H$ endowed with the operation 
$$xH*yH=xf(x^{-1}y)H$$ 
for every $x,y\in G$ is a quandle. We denote such quandle by $\mathcal{Q}(G,H,f)$ and refer to it as {\it coset quandles}. If $H=1$ we say that the coset quandle is {\it principal} and if $G$ is abelian, we say that it is {\it affine}. In the last case, we use the notation $\aff(G,f)$.

\subsection{Congruences and subgroups}

Let $Q$ be a quandle, $\alpha\in\con(Q)$ and $N$ be a normal subgroup of $\lmlt(Q)$. As introduced in \cite{CP}, we define the following equivalence relations and subgroups of the left multiplication group of $Q$: 
\begin{align*}
    \c{N} &= \{(x,y)\in Q\times Q\,\colon\, L_xL_y^{-1}\in N\}\in \con(Q); \\
    \mathcal{O}_N &= \{(x,y)\in Q\times Q\,\colon\, n(x)=y\,\text{ for some }n\in N\}\in \con(Q);\\
    \lmlt^\alpha &= \{h\in\lmlt(Q)\,\colon\, h(x)\,\alpha\, x\,\text{ for all } x\in Q\}\trianglelefteq \lmlt(Q); \\
    \dis_\alpha &= \langle L_xL_y^{-1}\colon\, x\,\alpha\, y\rangle \trianglelefteq \lmlt(Q).
\end{align*}
In particular, we define the {\it displacement group} of $Q$ as $\dis(Q)=\dis_{1_Q}$ which has the property that $\lmlt(Q)=\dis(Q)\langle L_x\rangle$ for every $x\in Q$ and consequently the orbits of $\lmlt(Q)$ and of $\dis(Q)$ coincide \cite[Proposition 2.1]{hsv}.

Note that $\lmlt^\alpha$ can be defined in the very same way for left quasigroups and $\dis_\alpha$ can also be defined as the normal closure of the subgroup generated by $\setof{L_x L_y^{-1}}{x\, \alpha \, y}$ in the left multiplication group. In particular, note that $\alpha\leq \lambda_Q$ if and only if $\dis_\alpha=1$. If so, we say that $Q$ is a {\it cover} of $Q/\alpha$ (see \cite{Eisermann, covering_paper, Involutory} for further details on quandle covers). We can also define $\dis^\alpha=\dis(Q)\cap \lmlt^\alpha$, and it is easy to verify that 
\begin{align*}
    \dis_\alpha\dis_\beta&=\dis_{\alpha\vee \beta} \\
    \dis^{\alpha\wedge\beta}&=\dis^\alpha\cap \dis^\beta\\
    \c{N\cap M}&=\c{N}\wedge\c{M}\\
    \mathcal{O}_{NM}&=\mathcal{O}_N\circ \mathcal{O}_M=\mathcal{O}_N\vee\mathcal{O}_M
\end{align*}
where $\alpha,\beta\in \con(Q)$ and $N,M$ are normal subgroups of $\lmlt(Q)$ (see also \cite[Proposition 3.2]{CP} for more details).

Moreover we have that $\mathcal{O}_N\leq \c{N}$ and according to \cite[Proposition 3.5(1)]{CP} we also  have that $$\dis_{\mathcal{O}_N}\leq \dis_{\c{N}}\leq N\leq \dis^{\mathcal{O}_N}\leq \dis^{\c{N}}$$  for every normal subgroup $N$ of $\lmlt(Q)$. In general we have
$$\mathcal{O}_{\dis_\alpha}\leq \mathcal{O}_{\dis^\alpha}\leq \alpha\leq \c{\dis_\alpha}\leq \c{\dis^\alpha}$$
for every congruence $\alpha$, and for finite latin quandles all such congruences are equal \cite[Lemma 1.8]{Galois}.

The pair of operators $(\dis_*,\c{*})$ (resp. $(\dis^*,\mathcal{O}_*)$) provides a monotone Galois connection between the congruence lattice of a quandle and the lattice of normal subgroups of its left multiplication group. We refer to \cite{CP} for more details on the properties of such Galois connections that we are going to exploit in the paper.

Given a left quasigroup $Q$ and $\alpha\in \con(Q)$, the map
$$\pi_\alpha:\lmlt(Q)\longrightarrow \lmlt(Q/\alpha),\quad L_x\mapsto L_{[x]}$$
can be extended to a surjective group homomorphism with kernel $\lmlt^\alpha$. Clearly $\pi_\alpha$ restricts and co-restricts to the displacement groups, and the kernel of this restriction is $\dis^\alpha$.

The displacement group plays a very important role for quandles. In fact, it can be used to build connected quandles and its factors with the coset quandle construction over the displacement group and its quotients (see also \cite{hsv}, Theorem 4.1, and \cite{J}). 

\begin{proposition}[\protect{\cite[Lemma 1.4]{GB}}]
    Let $Q$ be a connected quandle, $\alpha\in \con(Q)$ and $x\in Q$. Then $Q/\alpha\cong \mathcal{Q}(\dis(Q)/\dis^\alpha,\dis(Q)_{[x]}/\dis^\alpha ,\left(\widehat{L_x}\right)_{\dis^\alpha})$, where $\dis(Q)_{[x]}=\setof{h\in \dis(Q)}{h(x)\, \alpha\, x}$.
\end{proposition}

The coset quandle construction over the displacement group allows also to identify the isomorphism classes of quandles. Indeed, we have the following isomorphism theorem. 
\begin{theorem}[\protect{\cite[Theorem 5.12]{GB}\cite[Corollary 3.26]{GiuThe}}] \label{iso theo}
    Let $G$ be a group, $f_i\in \aut{G}$ and let $Q_i=\mathcal{Q}(G,Fix(f_i),f_i)$ be such that $\dis(Q_i)=G$ for $i=1,2$. Then $Q_1\cong Q_2$ if and only if $f_1$ and $f_2$ are conjugate in $\aut{G}$.
\end{theorem}

Sylow subgroups of nilpotent normal subgroups of the left multiplication group of a quandle provide trivially intersecting congruences through the operator $\mathcal{O}_*$.
\begin{lemma}\label{Sylow}
    Let $Q$ be a finite connected faithful quandle, $N$ be a normal nilpotent subgroup of $\lmlt(Q)$ such that $N=\prod_i P_i$ where $P_i$ are the $p$-Sylow subgroups of $N$. Then $ \mathcal{O}_{P_i}\wedge \mathcal{O}_{P_j}=0_Q$ for all $i\neq j$ and $\circ_{i} \mathcal{O}_{P_i}=\mathcal{O}_N$.
\end{lemma}

\begin{proof}

The Sylow subgroups of $N$ are characteristic in $N$, so they are normal in $\lmlt(Q)$. Thus, $\mathcal{O}_{P_i}$ is a congruence of $Q$ for every $i$. Moreover $\mathcal{O}_{P_i}\wedge \mathcal{O}_{P_j}\leq \c{P_i}\wedge \c{P_j}=\c{P_i\cap P_j}=\c{1}=\lambda_Q=0_Q$, since $P_i\cap P_j=1$. 
%
     Moreover $N=\prod_i P_i$ and so $\mathcal{O}_N=\circ_i \mathcal{O}_{P_i}$.
\end{proof}

\begin{corollary}\label{center of dis^alpha}
    Let $Q$ be a finite faithful quandle, and $\alpha\in \ma_Q$. The following are equivalent:
    \begin{itemize}
        \item[(i)] $Z(\dis^\alpha)\neq 1$.
        \item[(ii)] $\dis^\alpha$ is a $p$-group of nilpotency class less than or equal to $2$ for some prime $p$.
    \end{itemize}
\end{corollary}
\begin{proof}
    (ii) $\Rightarrow$ (i) Clear.

    (i) $\Rightarrow$ (ii) Let $Z=Z(\dis^\alpha)\neq 1$. The subgroups $Z$ is characteristic in $\dis^\alpha$ and so is normal in $\lmlt(Q)$. Then $\beta=\mathcal{O}_Z\leq \alpha$ is a proper congruence, and so $\alpha=\beta$. Hence $\dis^\alpha\leq \dis_{\c{Z}}\leq Z$. According to \cite[Propositon 3.3(1)]{CP} we have $[\dis^\alpha,\dis^\alpha]\leq \dis_\alpha\leq Z$ and so $\dis^\alpha$ has nilpotency class $2$. According to Lemma \ref{Sylow} the Sylow subgroups of $\dis^\alpha$ provide trivially intersecting congruence below $\alpha$. Since $\alpha$ is minimal, we see that $\dis^\alpha$ is a $p$-group for some prime $p$.
\end{proof}

In the final part of this section, we examine how subdirect reducibility and direct decomposability of left quasigroups and quandles relate to the corresponding properties of their left multiplication group and displacement group.
\begin{lemma}\label{DD LQGs}
Let $Q_1$ and $Q_2$ be left quasigroups. Then $\dis(Q_1\times Q_2)\cong \dis(Q_1)\times \dis(Q_2)$.
\end{lemma}
\begin{proof}
    Let $p_i:Q\longrightarrow Q_i$ be the canonical projection, and $\alpha_i=\ker{p_i}$ for $i=1,2$. Then $\dis^{\alpha_1}\cap \dis^{\alpha_2}=\dis^{\alpha_1\wedge\alpha_2}=\dis^{0_Q}=1$ and $\dis(Q)=\dis_{\alpha_1\vee \alpha_2}=\dis_{\alpha_1}\dis_{\alpha_2}\leq \dis^\alpha_1\dis^\alpha_2$. Hence $\dis(Q)\cong \dis^\alpha_1\times \dis^\alpha_2$. Since $\dis(Q)/\dis^\alpha_1\cong \dis(Q/\alpha_1)\cong \dis^{\alpha_2}$ and $\dis(Q)/\dis^\alpha_2\cong \dis(Q/\alpha_2)\cong \dis^{\alpha_1}$ the statement follows.
\end{proof}
\begin{corollary}
    Let $Q$ be a directly decomposable left quasigroup. Then $\dis(Q)$ is directly decomposable.
\end{corollary}

Note that if $\dis(Q)$ is directly decomposable, then $Q$ does not need to be directly decomposable, not even for quandles. In fact $Q={\tt SmallQuandle(32,5)}$ is directly indecomposable, but $\dis(Q)\cong \mathcal{Q}_8\times \Z_2^2$.

\begin{lemma}
Let $Q$ be a connected faithful quandle. 
If $\lmlt(Q)$ is directly decomposable, then $Q$ is directly decomposable.  
\end{lemma}

\begin{proof}


Assume that $\lmlt(Q)$ is directly decomposable, that is, there are normal subgroups $1 \neq N,M\trianglelefteq\lmlt(Q)$ such that $N\cap M=1$ and $\lmlt(Q)=NM$. Then $0_Q\neq \mathcal{O}_N\leq \c{N}$, $0_Q\neq \mathcal{O}_M\leq \c{M}$ and $\c{N}\wedge\c{M}=\c{N\cap M}=\c{1}= \lambda_Q=0_Q$. Moreover $\c{N}\circ \c{M}\geq \mathcal{O}_N\circ \mathcal{O}_M=\mathcal{O}_{NM}=\mathcal{O}_{\lmlt(Q)}=1_Q$. Then $Q$ is directly decomposable.
\end{proof}

For faithful left quasigroups, subdirect irreducibility can also be reduced to the same property for the left multiplication group.

\begin{lemma}\label{SI LQGs}
Let $Q$ be a faithful left quasigroup. 
\begin{itemize}
    \item[(i)] If $Q$ is subdirectly reducible, then $\lmlt(Q)$ is subdirectly reducible. 
\item[(ii)]  If $Q$ is a quandle and $\lmlt(Q)$ is subdirectly reducible, then $Q$ is subdirectly reducible.
\end{itemize}
\end{lemma}

\begin{proof}
(i) Assume that $Q$ is subdirectly reducible, that is, there exists a set $\setof{\alpha_i}{i\in I}\subseteq \con(Q)$ such that $\alpha_i\neq 0_Q$ for all $i\in I$ and $\beta=\bigwedge_{i\in I} \alpha_i=0_Q$. The left quasigroup $Q$ is faithful, so $\dis_{\alpha_i}\neq 1$ for every $i\in I$, so also the group $\dis^{\alpha_i}$ is not trivial for every $i\in I$. Then $\bigcap_{i\in I} \dis^{\alpha_i}=\dis^\beta=1$ and so $\lmlt(Q)$ is subdirectly reducible. 
%
%


(ii) Assume that $\lmlt(Q)$ is subdirectly reducible, that is, there exist $N,M\neq 1$ be normal subgroups of $\lmlt(Q)$ such that $N\cap M=1$. Then $0_Q\neq \mathcal{O}_N\leq \c{N}$, $0_Q\neq \mathcal{O}_M\leq \c{M}$ and $\c{N}\wedge\c{M}=\c{N\cap M}=\c{1}= \lambda_Q=0_Q$. Then $Q$ is subdirectly reducible.
\end{proof}


\subsection{Commutator theory for quandles}

Commutator theory for quandles can be understood in purely group-theoretical terms by looking at the properties of the displacement group and its subgroups. The main results are the following, and they will be freely used throughout the paper. Let $Q$ be a quandle:
\begin{itemize}
    \item 
    if $Q$ is faithful and $\alpha\in \con(Q)$ then $\alpha$ is abelian (resp. central) if and only if $\dis_\alpha$ is abelian (resp. central in $\dis(Q)$) \cite[Corollary 5.4]{CP}.

\item The quandle $Q$ is solvable (resp. nilpotent) if and only if $\dis(Q)$ is solvable (resp. nilpotent) \cite[Theorem 1.2]{CP}.



\item Affine quandles are abelian. On the other hand, connected abelian quandles are affine \cite[Corollary 3.4]{Medial}.


\item If $Q$ is finite and connected, then $\gamma_Q=\mathcal{O}_{\gamma_1(\dis(Q))}$ and $\dis^{\gamma_Q}=\gamma_1(\dis(Q))$ \cite[Proposition 1.7]{GB}.

%
 


\end{itemize}

For finite latin quandles, we have the following characterization of the center and its relative displacement group.
\begin{lemma}\label{center for latin}
    Let $Q$ be a finite latin quandle. Then $\zeta_Q=\mathcal{O}_{Z(\dis(Q))}$ and $\dis_{\zeta_Q}=Z(\dis(Q))$.
\end{lemma}

\begin{proof}
Recall that for finite latin quandles we have that the operators $\mathcal{O}_*$ and $\c{*}$ coincide. According to \cite[Corollary 5.10]{CP} we have $\zeta_Q=\c{Z(\dis(Q))}$ and so it coincides with $\mathcal{O}_{Z(\dis(Q))}=\mathcal{O}_{\dis_{\zeta_Q}}$ and $\dis_{\zeta_Q}\leq Z(\dis(Q))$ accordingly. The subgroup $Z(\dis(Q))$ is a central subgroup of a transitive group, and so $Z(\dis(Q))_x=1$ for every $x\in Q$. Therefore, $Z(\dis(Q))=\dis_{\zeta_Q}Z(\dis(Q))_x=\dis_{\zeta_Q}$.
\end{proof}

The blocks of congruences of quandles are subquandles, and the blocks of abelian congruences are abelian quandles (this follows directly from the abstract definition of abelianness of a congruence as defined in \cite{comm}). We can use the description of abelian quandles and the close relationship between the properties of congruences and the properties of subgroups of the displacement group to understand the features of abelian congruences (or more generally congruences with abelian blocks) to build some tools to study solvable quandles.

Under some further assumptions, the blocks of abelian congruences are affine, so we have the following description of their structure.

\begin{lemma}\label{below abelian}
    Let $Q$ be a finite latin quandle, and let $\alpha$ be an abelian congruence. Then:
    \begin{itemize}
        \item[(i)] The block $[a]_\alpha$ is affine, i.e $[a]_\alpha\cong \aff(A,f)$ for some abelian group $A$ and some $f\in \aut{A}$. 
        \item[ (ii)] If $\beta\leq\alpha$ then $[0]_\beta$ is an $f$-invariant subgroup of $A$. In particular, $|\setof{\beta\in \con(Q)}{\beta\leq \alpha}|\leq |\setof{N\leq A}{f(N)=N}|$.
    \end{itemize}

\end{lemma}

\begin{proof}
%
(i) The blocks of $\alpha$ are abelian latin quandles. So, they are affine, that is, $[x]_\alpha\cong \aff(A,f)$ for some abelian group $A$ and some $f\in \aut{A}$. 

(ii) Given $\beta\leq \alpha$ then $[0]_\beta$ is a subquandle of $\aff(A,f)$ and so it is an $f$-invariant subgroup of $A$ \cite[Proposition 3.14]{Principal}. Connected quandles are congruence regular (see \cite[Lemma 2.4]{Maltsev_paper}), so the map $\beta \mapsto [0]_\beta $ from $\setof{\beta\in \con(Q)}{\beta\leq \alpha}$ to $\setof{N\leq A}{f(N)=N}$ is injective.
\end{proof}

We are particularly interested to understand minimal congruences of quandles, in particular for solvable ones (e.g. finite latin quandles, see Corollary 1.3 of \cite{CP}).

\begin{lemma}\label{minimal cong}
Let $Q$ be a finite faithful quandle, and $\alpha\in \ma_Q$. 
\begin{itemize}
    \item[(i)] $\alpha=\mathcal{O}_{\dis_\alpha}$ and $\dis_\alpha$ is a minimal normal subgroup of $\lmlt(Q)$.
    \item[(ii)] If $\alpha$ is abelian, then $\dis_\alpha$ is an elementary abelian $p$-group for some prime $p$ and the blocks of $\alpha$ have size a power of $p$.
\item[(iii)]  If $Q$ is solvable, then $\alpha$ is abelian. In particular, $\nu_Q$ is abelian.
\end{itemize}

\end{lemma}
\begin{proof}

(i) Let $\beta=\mathcal{O}_{\dis_\alpha}\leq \alpha$. So, $\beta=0_Q$ or $\alpha=\beta$. If $\beta=0_Q$ then $\dis_\alpha=1$. Thus, $\alpha\leq \lambda_Q=0_Q$, contradiction.

Assume that $N\leq \dis_\alpha$. Then $\delta=\mathcal{O}_N\leq \alpha$. Hence $\delta=0_Q$ and accordingly $N=1$ or $\delta=\alpha$ and so $\dis_\alpha\leq N$ (see \cite{CP}, Section 3.2). 

(ii) According to \cite[Corollary 5.4]{CP} the subgroup $\dis_\alpha$ is abelian. Let $H$ be a characteristic subgroup of $\dis_\alpha$. Then $H$ is normal in $\lmlt(Q)$ and therefore trivial since $\dis_\alpha$ is a minimal normal subgroup of $\lmlt(Q)$. Hence $\dis_\alpha$ is necessarily an elementary abelian $p$-group for some prime $p$ (the $p$-Sylow subgroup and the subgroups of the form $p^n\dis_\alpha$ for $n\in \mathbb{N}$ are characteristic). So, $|[x]|=[\dis_\alpha:(\dis_\alpha)_x]$ is a power of $p$.

(iii) The group $\dis(Q)$ is solvable, since $Q$ is solvable. Let $\alpha$ be a minimal congruence. Then $\dis_\alpha$ is a solvable minimal normal group. Therefore, $[\dis_\alpha,\dis_\alpha]<\dis_\alpha$, and thus $[\dis_\alpha,\dis_\alpha]=1$. Hence $\dis_\alpha$ is abelian and so is $\alpha$. According to \cite[Proposition 3.2(2)]{CP} we have 
$$\dis_{\nu_Q}=\dis_{\vee \{\alpha\in \ma_Q\}}=\langle\dis_\alpha,\,\alpha \in \ma_Q\rangle.$$
Moreover, $[\dis_\alpha,\dis_\beta]\leq \dis_\alpha\cap\dis_\beta=1$ and the subgroup $\dis_\alpha$ is abelian for every $\alpha,\beta\in\ma_Q$. Hence $\dis_{\nu_Q}$ is abelian and so $\nu_Q$ is also abelian. 
\end{proof}





\begin{lemma}\label{dis_gamma}
Let $Q$ be a non-nilpotent quandle. Then: 
\begin{itemize}
  \item[(i)] If $Q$ is faithful and $\gamma_Q\in \ma_Q$, then $\dis_{\gamma_Q}=\gamma_2(\dis(Q))$. 
    \item[(i)] If $\ma_Q=\{\gamma_Q\}$, then $Z(\dis(Q))=1$.
  
\end{itemize}
\end{lemma}

\begin{proof}
Let $G=\dis(Q)$.

(i) According to \cite[Proposition 3.3(1)]{CP} we have $\gamma_2(G)=[G,\gamma_1(G)]\leq \dis_{\gamma_Q}\neq 1$ and Lemma \ref{minimal cong}(i) implies that the subgroup $\dis_{\gamma_Q}$ is a minimal normal subgroup of $\lmlt(Q)$. So, $\gamma_2(G)=1$ or $\gamma_2(G)=\dis_{\gamma_Q}$. In the first case, $G$ is nilpotent and so also $Q$ is nilpotent by \cite[Theorem 1.2(2)]{CP}, contradiction. Then $\dis_{\gamma_Q}=\gamma_2(G)$.

(ii) Assume that $Z(G)\neq 1$. Then $\beta=\mathcal{O}_{Z(G)}\neq 0_Q$. According to \cite[Lemma 5.12]{CP} $\beta$ is central and $\gamma_Q\leq \beta$ since $\gamma_Q$ is the only minimal congruence of $Q$. Thus, $\gamma_2(Q)=[\gamma_Q,1_Q]=0_Q$ and so $Q$ is nilpotent, which is a contradiction.
\end{proof}


We now introduce a corollary of Lemma 1.4 of \cite{LSS}.
\begin{corollary}
    \label{embedding as quandle}
	Let $Q$ be a connected quandle, $\alpha\in \con(Q)$ such that $Q/\alpha=Sg([a_1],\ldots, [a_n])$. Then:
	\begin{itemize}
		\item[(i)] the mapping
		\begin{displaymath}\psi: \dis^\alpha\longrightarrow \prod_{i =1}^n \aut{[a_i]},\quad h\mapsto (h|_{[a_1]},\ldots, h|_{[a_n]}),
		\end{displaymath}
		is an embedding of groups and $\dis^\alpha|_{[a_i]_\alpha}\cong \dis^\alpha|_{[a_j]_\alpha}$ for every $1\leq i,j\leq n$.
		
		\item[(ii)] If $Q$ is a finite latin quandle and $\alpha$ is abelian, then $\dis_\alpha|_{[a]}\cong \dis([a])$ and $\dis_\alpha$ embeds in $\dis([a])^n$. In particular $|\dis_\alpha|$ divides $|[a]|^n$. 
	\end{itemize}

\end{corollary}

 \begin{proof}
	(i) Apply Lemma 1.4(i) of \cite{LSS} to $\dis^\alpha$.
	
	(ii) The blocks of $\alpha$ are finite latin affine quandles. The group $\dis_\alpha$ is abelian and so $\dis_\alpha|_{[x]}$ is regular on $[x]$. Then Lemma 1.4(ii) of \cite{LSS} applies to $\dis_\alpha$. So $\dis_\alpha$ embeds in $\dis([a])^n$ that has size $|[a]|^n$.  
\end{proof}

In the following, we are displaying some sufficient conditions for a congruence of a quandle to be central. 
\begin{lemma}\label{central 2}
    Let $Q$ be a connected faithful quandle and $\alpha\in \con(Q)$. If $\dis_\alpha$ is cyclic, then $\alpha$ is central.
\end{lemma}

\begin{proof}
According to \cite[Proposition 2.3]{hsv} we have $\dis(Q)=[\lmlt(Q),\lmlt(Q)]$. Assume that $H=\dis_\alpha$ is cyclic. The subgroup $H$ is normal in $\lmlt(Q)$, so we have a group action $\rho:\lmlt(Q)\longrightarrow \aut{H}$ where $\rho_x=\widehat{x}|_H$ for every $x\in \lmlt(Q)$. The group $\aut{H}$ is abelian, so $\dis(Q)=[\lmlt(Q),\lmlt(Q)]\leq \ker{\rho}$, i.e. $H$ is central in $\dis(Q)$. According to \cite[Corollary 5.4]{CP}, we see that $\alpha$ is a central congruence.
\end{proof}
Note that Lemma \ref{central 2} only uses that the automorphism group of $\dis_\alpha$ is abelian, so if this condition holds, then $\alpha$ is central.
%


%
%

\begin{lemma}\label{central}
Let $Q$ be a finite connected faithful quandle and let $\alpha\in \con(Q)$ such that $\alpha\wedge \gamma_Q=0_Q$. Then: 
\begin{itemize}
\item[(i)] $\alpha$ is central; 
\item[(ii)] $\dis_\alpha=\dis^\alpha$; 
\item[(iii)] $Q/\alpha$ is faithful.
\end{itemize}

\end{lemma}

\begin{proof}
(i) Let $A$ be an algebra, and $\alpha$ be a congruence. Then $[1_A,\alpha]\leq \alpha\wedge \gamma_A=0_A$ \cite[Proposition 3.4]{comm}. Therefore, $\alpha$ is central.

(ii) According to \cite[Proposition 3.3(i)]{CP} we have
$$[\dis(Q),\dis^\alpha]\leq \gamma_1(\dis(Q))\cap \dis_\alpha=\dis^{\gamma_Q}\cap \dis_\alpha\leq \dis^{\gamma_Q}\cap \dis^\alpha=\dis^{\gamma_Q\wedge \alpha}=\dis^{0_Q}=1.$$
Then $\dis^\alpha\leq Z(\dis(Q))$. According to \cite[Lemma 4.4(i)]{Nilpotent}, $\dis^\alpha=\dis_\alpha\times \left(\dis^\alpha\right)_x$ and since $\dis^\alpha$ is a central subgroup of the transitive group $\dis(Q)$ then $\left(\dis^\alpha\right)_x=1$. Consequently $\dis_\alpha=\dis^\alpha$. 

(iii) Using \cite[Corollary 4.5(i)]{Nilpotent}, we have $\alpha=\c{\dis_\alpha}=\c{\dis^\alpha}$. Then we can apply \cite[Proposition 3.7]{CP} and we find that $Q/\alpha$ is faithful.
\end{proof}

\section{latin quandles of size \texorpdfstring{$16p$}{16p}} \label{sec:latin16p}


\subsection{Locally strictly simple quandles}
A quandle is {\it strictly simple} if it does not have proper subquandles. Finite strictly simple quandles coincide with simple latin quandles \cite{Principal}. Latin quandles are solvable (see \cite{CP}, Corollary 1.3), thus finite strictly simple quandles are abelian and hence affine. Recall that finite simple abelian quandles are given by affine quandles of the form $Q=\Aff(\Z_p^n,f)$ where $\Z_p^n$ does not have $f$-invariant subgroups (see \cite{Principal} and \cite{Medial} for more details on abelian quandles). In particular, $\aut{Q}\cong \Z_p^n\rtimes \Z_{p^n-1}$, $\aut{[x]}_x= C_{GL_n(p)}(f)\cong \Z_{p^n-1}$ and $\aut{Q}/\dis(Q)$ is abelian \cite[Proposition 3.8]{Principal}. 

The following result has already been implicitly stated in the proof of \cite[Theorem 3.14]{LSS} for which we include a cleaner proof divided into two parts. 
\begin{lemma}\label{aut of simple}
    Let $Q=\aff(\Z_p^n, f)$ be a simple quandle, and $g\in \aut{Q}$. If $|g|$ and $p$ are relatively prime, then $g\in \aut{Q}_x$ for some $x\in Q$. 
\end{lemma}

\begin{proof}
   If $g\in \aut{[x]}$ has order relatively prime with $p$ then $g(y)=t+h(y)$ for $1\neq h\in C_{GL_n(p)}(f)$ and some $t\in \Z_p^n$ for every $y\in Q$. Assume that $0\neq y\in Fix(h)$. Then $f(y)=fh(y)=h(f(y))$, so $Fix(h)$ is an $f$-invariant subgroup. 
   Therefore, $Fix(h)=0$ and so $1-h$ is invertible. Hence, there exists $z\in Q$ such that $(1-h)(z)=t$, i.e. $g(z)=z$.
\end{proof}

\begin{lemma}\label{N cyclic}
    Let $Q$ be a finite quandle, $\alpha\in \con(Q)$ and $N\leq \dis^\alpha$. If $[x]_\alpha$ is a simple affine quandle of order relatively prime with $|N|$ then $N$ embeds into $\aut{[x]}_x$ for some $x\in Q$. In particular $N$ is cyclic.
\end{lemma}

\begin{proof}
 Let $g\in N$. Then $g|_{[x]}\in\aut{[x]}$ has order relatively prime with $|[x]|$ and so for every $[x]\in Q/\alpha$ there exists $c\in [x]$ such that $g(c)=c$ (see Lemma \ref{aut of simple}). Let $X$ be any set of representatives of the blocks of $\alpha$. Then $Q$ is generated by $X\cup [x]_\alpha$ for every $x\in Q$ \cite[Lemma 6.1]{GB}.
%
 Let $\phi:g\mapsto g|_{[x]}$ map $N$ to $\aut{[x]}$. If $\phi(g)=1$ then $g$ fixes all the elements of block $[x]$ and an element for every other block of $\alpha$. Then $g$ fixes a set of generators of $Q$ and so $g=1$. Hence $N$ embeds into $\aut{[x]}$.

Since $\aut{[x]}/\dis([x])$ is abelian, then $[\aut{[x]},\aut{[x]}]\leq \dis([x])$ and so is also $[\phi(N),\phi(N)]=\phi([N,N])\leq \dis([x])$. The order of $N$ and of $\dis([x])$ are relatively prime (indeed $|\dis([x])|=|[x]|$), so $[\phi(N),\phi(N)]=1$ and thus $N$ is abelian. So, if $g,g'\in N$ and $g(x)=x$ then $g'(x)=g'(g(x))=g(g'(x))$. So $g$ fixes $S=Sg(x,g'(x))$. If $g'(x)\neq x$ then $S=[x]$ and $g=1$ accordingly. Therefore, $g'(x)=x$, so we have that $N$ embeds in $\aut{[x]}_x$ which is cyclic. So $N$ is also cyclic. 
\end{proof}

We now consider the class of {\it locally strictly simple} (LSS) quandles, that is, the class of quandles such that all proper subquandles are strictly simple. In \cite{LSS} we proved that the non-affine subdirectly irreducible LSS quandles have a unique proper congruence with strictly simple factor and strictly simple blocks (see \cite{LSS}). Let $\sigma_Q$ be the equivalence relation defined by setting
$$x\, \sigma_Q\, y \text{ if and only if }\dis(Q)_x=\dis(Q)_y.$$ 
We denote by $\LSS(p^m,q^n,\alpha)$ the class of subdirectly irreducible non-affine LSS quandles such that the unique factor has size $q^n$, the blocks of the unique congruence have size $p^m$ and $\sigma_Q=\alpha$.

We can use \cite[Proposition 5.7]{Principal} (also mentioned in \cite[Proposition. 3.11]{LSS}) to describe the structure of the displacement group of quandles with the congruence lattice given by the three-element chain that includes all LSS latin quandles. 

\begin{lemma}\label{dis for 3-chains}
    Let $Q$ be a finite connected faithful solvable quandle with congruence lattice given by a chain of length $3$. If $Q$ is not nilpotent then one of the following holds: 
    \begin{itemize}
        \item[(i)] $\gamma_1(\dis(Q))=\gamma_2(\dis(Q))$ and $\dis(Q)=\Z_p^k\rtimes_\rho \Z_q^m$ where $p,q$ are primes and $k\in \mathbb{N}$.  
        \item[(ii)] $\gamma_1(\dis(Q))\neq\gamma_2(\dis(Q))$ and $\dis(Q)=\Z_p^k\rtimes_\rho K$ where $p$ is a prime, $k\in \mathbb{N}$ and $K$ is a special $q$-group. 
    \end{itemize}
In both cases, $Z(\dis(Q))=1$ and $\rho$ is faithful.
\end{lemma}

\begin{proof}
Let $G=\dis(Q)$. Since $Q$ is solvable but not nilpotent, then the unique non-trivial congruence of $Q$ is $\gamma_Q$, it is abelian and the blocks have size a power of a prime (see Lemma \ref{minimal cong}). According to Lemma \ref{dis_gamma}(ii) we have $\dis_{\gamma_Q}=\gamma_2(G)$ and it is an elementary abelian $p$-group according to Lemma \ref{minimal cong}(ii). The quandle $Q/\gamma_Q$ is simple and abelian, therefore $\dis(Q/\gamma_Q)=G/\gamma_1(G)$ is elementary abelian. So $Q/\gamma_Q\cong \aff(\Z_q^n,f)$ where $f$ has no invariant subgroups. The quandle $Q$ is not nilpotent, so $p\neq q$ (connected quandles of size a power of a prime are nilpotent \cite[Theorem 1.4]{CP}).

Let $\gamma_1(G)\leq H\neq G$ where $H$ is a characteristic subgroup of $G$. Then $H/\gamma_1(G)$ is an $f_{\gamma_1(G)}$-invariant subgroup of $G/\gamma_1(G)$, so $H=\gamma_1(G)$ and $\gamma_1(G)$ is a maximal characteristic subgroup of $G$. Thus, the group $G$ satisfies the hypothesis of \cite[Proposition 5.7]{Principal} and the statement follows. 

According to Lemma \ref{dis_gamma}(i), $\dis(Q)$ has trivial center, and therefore $\rho$ is faithful according to Lemma \ref{no center}. 
\end{proof}

For the scope of this paper, we need to investigate the classes $\LSS(p,16,\sigma_Q)$ and $\LSS(16,p,\sigma_Q)$ as they are classes of quandles of size $16p$.

\begin{proposition}\label{no LSS}
    The class $\LSS(p,16,\alpha)$ is empty for every $\alpha$.
\end{proposition}

\begin{proof}
  Let $Q\in \LSS(p,16,\alpha)$ and $G=\dis(Q)$.  According to the results of Section 3 of \cite{LSS} and following the argument of the proof of Proposition 3.21 of \cite{LSS} it follows that $G=\Z_p^2\rtimes_\rho K$ where $K$ is an extraspecial $2$-group of size $32$, $\rho$ is faithful, $\gamma_1(G)= \Z_p^2\rtimes_\rho Z(K)$ and $Z(\gamma_1(G))=1$. There are only $2$ extraspecial $2$-groups of size $32$, i.e. $K_i=$ \texttt{SmallGroup(32,i)} for $i=49,50$. If $K=K_{49}$, then it has a subgroup isomorphic to $\Z_2^3$. Thus, $\rho$ is not faithful according to Corollary \ref{ro faithful on Z_p}. If $K=K_{50}$, then $\gamma_1(G)$ has non-trivial center, according to Proposition \ref{no rep}. So, there are no quandles in the class $\LSS(p,16,\sigma_Q)$.
\end{proof}


\subsection{Factors of latin quandles of size \texorpdfstring{$16p$}{16p}}

The size of factors of latin quandles of size $16p$ is limited according to the known results on latin quandles of size dividing $16p$. In this subsection, we identify the size of the factors of latin quandles of size $16p$ and we also recall the known results about such quandles.

\begin{lemma}\label{meet of max is min}
Let $Q$ be a latin quandle of size $16p$. Then:

\begin{itemize}

\item[(i)] $|Q/\alpha|\in \{1,4,16,p,4p,16p\}$ for every $\alpha	\in \con(Q)$.	

\item[(ii)] $Q$ is congruence slim.


 
 
\end{itemize}

\end{lemma}

\begin{proof}
(i) The size of $Q/\alpha$ divides the size of $Q$, so $|Q/\alpha|=2^k p^s$ for $0\leq k\leq 4$ and $0\leq s\leq 1$. The factors or congruence blocks of $Q$ cannot have size $2$ or $2p$ because there are no connected quandles for such orders for $p>5$ (see \cite{McCarron}) and there are no latin quandles of size $6$ and $10$ (see the database of connected quandles \cite{RIG} in GAP). 

(ii) Let $\alpha\in \ma_Q$. Then according to (i), we have $|Q/\alpha|\in \{4,16,p,4p\}$. In all cases, the height of $Q/\alpha$ is less than or equal to $3$ and therefore the height of $Q$ is less than or equal to $4$. 
%
%
%
%
%
%
\end{proof}
According to Lemma \ref{meet of max is min}(ii) we can apply the results of Section \ref{slim} to latin quandles of size $16p$, and so we know how their congruence lattices look like.

Let us list the latin quandles of size $\{4,16,p,4p\}$ where $p$ is an odd prime. 
\begin{itemize}
    \item[(i)] We denote by $\Q_4$ the unique latin quandle of order $4$. In particular $\Q_4$ is affine over $\Z_2^2$ and it is simple.

\item[(ii)] Latin quandles of size $16$ are affine over elementary abelian groups $\Z_2^4$ \cite{RIG}.

\item[(iii)] Latin quandles of size $p$ are affine over cyclic groups of prime size \cite{EGS}. 


 \item[(iv)] Latin quandles of size $4p$ with $p=3,5$ are affine \cite{RIG}. 
 Latin quandles of size $4p$ for $p>5$ are affine or are locally strictly simple, and they are completely classified in \cite{LSS}. The affine quandles of size $4p$ are isomorphic to $\mathcal{Q}_4\times \aff(\Z_p,f)$ and are generated by every pair $(x,0),(0,y)$ for $x,y\neq 0$. We denote the non-affine latin quandles of size $4p$ by $Q(p,j)$ for $j=1,2$. Such quandles are in the class $\LSS(p,4,0_Q)$ and exist only when $p=1\pmod{3}$. Let $k\in \Z_p$ such that $k^3=1$ and $k\neq 1$. The coset representation of $Q(p,j)$ is $\mathcal{Q}(\mathcal{G}_k,Fix(f(j)),f(j))$ where $\mathcal{G}_k\cong \Z_p^2\rtimes_\rho \mathcal{Q}_8$ is the group defined in Subsection \ref{grup for sr}, where
 $$\mathcal{Q}_8=\langle x,y,z\,|\, x^2,\, y^2,\, [x,y]z^{-1},\, z^2,\,[z,y],\,[z,x]\rangle,$$
 the action $\rho$ is defined on the generators of $\mathcal{Q}_8$ by setting 
\begin{displaymath}
\rho_x=\begin{bmatrix}
0 & -1\\
1 & 0
\end{bmatrix},\quad
\rho_{y}=\begin{bmatrix}
k^2 & k\\
k & -k^2
\end{bmatrix},\quad \rho_z=\begin{bmatrix}
-1 & 0\\
0 & -1
\end{bmatrix},
\end{displaymath}
 and 
\begin{align}\label{fj}
      f(j)=\begin{cases}
          x\mapsto y,\\
                  y\mapsto xy,\\
                  z\mapsto z\\   f|_{\gamma_2(\mathcal{G}_k)}=\begin{bmatrix}
                      -k(1+k^j) & -(1+k^j)\\
                      0 & -k^2(1+k^j)
                  \end{bmatrix}
                  \end{cases}
\end{align}
 for $j=1,2$.

\end{itemize}

In the paper, we are also using the classification of non-faithful connected quandles of size $32$ that can be found in \cite{RIG}. We are collecting the data we will use later in Table \ref{32}, in particular, we list the connected non-faithful quandles of size $32$ such that the abelianization of the displacement group is the elementary abelian $2$-group of size $16$.


\begin{table}[ht!]
    \centering
    \begin{tabular}{|c|c|c|c|}
\hline
$Q$   & $\dis(Q)$ &  $|\con(Q/\gamma_Q)|$  \\
\hline
\text{{\tt SmallQuandle(32,2)}} & \text{{\tt SmallGroup(32,47)}}  &  7   \\
\text{{\tt SmallQuandle(32,3)}} & \text{{\tt SmallGroup(32,49)}}   &   7  \\
\text{{\tt SmallQuandle(32,4)}} & \text{{\tt SmallGroup(32,50)}}   &   2  \\
\text{{\tt SmallQuandle(32,5)}} & \text{{\tt SmallGroup(32,47)}}   &   3  \\
\text{{\tt SmallQuandle(32,6)}} & \text{{\tt SmallGroup(32,49)}}    &   3  \\
\hline 
    \end{tabular}
    \caption{Non-affine connected quandles of size $32$ such that the largest affine factor is affine over the group $\Z_2^4$.}
    \label{32}
\end{table}


\subsection{Subdirectly reducible case}

Latin quandles of size $4,16,p$ and $4p$ are classified, so we can focus on directly indecomposable latin quandles of size $16p$, where $p$ is an odd prime.

\begin{lemma}\label{16p nilpotent are abelian}
Let $Q$ be a latin quandle of size $16p$. The following are equivalent: 
\begin{itemize}
    \item[(i)] $Q$ is nilpotent.
  \item[(ii)] $Q$ is abelian.
\end{itemize}
Moreover, if $Q$ is nilpotent, then $Q$ is directly decomposable. 
\end{lemma}

\begin{proof}
Abelian quandles are nilpotent. Assume that $Q$ is nilpotent. Finite latin nilpotent quandles are direct products of prime power order quandles \cite[Theorem 1.4]{CP}. Latin quandles of size $16$ and $p$ are abelian, therefore so is their direct product. Hence $Q$ is abelian and directly decomposable.
\end{proof}


\begin{proposition}\label{Sub red}
Let $Q$ be a subdirectly reducible latin quandle of size $16p$. If $Q$ is directly indecomposable then 
%
    %
    $\con(Q)$ is the following:
\begin{align}\label{cong of not SI}
    \xymatrixrowsep{0.15in}
		\xymatrixcolsep{0.15in}
		\xymatrix{ 
			 & 1_Q  \ar@{-}[d]^4& &\\
			 & \nu_Q \ar@{-}[dr]^p\ar@{-}[dl]_4 & &\\
			\gamma_Q \ar@{-}[dr]_p  & & \zeta_Q\ar@{-}[dl]^4 \\
			 & 0_Q&  &
		}
\end{align}   

\end{proposition}

\begin{proof}
Let $Q$ be a subdirectly reducible latin quandle of size $16p$. If $Q$ is directly indecomposable then the congruence lattice of $Q$ is as follows according to Proposition \ref{lattices}(ii):
\begin{center}
		$\xymatrixrowsep{0.15in}
		\xymatrixcolsep{0.15in}
		\xymatrix{ 
			& & 1_Q  \ar@{-}[d]& &\\
			& & \nu_Q  \ar@{-}[d]\ar@{-}[dr]\ar@{-}[drr] \ar@{-}[dl] \ar@{-}[dll] & &\\
	\alpha_1 \ar@{-}[drr] & \alpha_2\ar@{-}[dr] &\ldots\ar@{-}[d] &  \alpha_{n-1}\ar@{-}[dl] & \alpha_n \ar@{-}[dll]\\
			& & 0_Q&  &
		}$
	\end{center}
where $n\geq 2$. Moreover, $|Q/\nu_Q|\in \{4,16,p\}$ since finite simple latin quandles have size a power of a prime (as they are of the form $\aff(\Z_p^n,f)$). Let us proceed case by case according to the size of $Q/\nu_Q$.
\begin{itemize}
    \item If $|Q/\nu_Q|=16$, then $\nu_Q$ has blocks of prime size and so $\nu_Q$ is minimal. Thus, $Q$ is not subdirectly reducible.

\item If $|Q/\nu_Q|=p$, then $Con(Q/\alpha_i)$ is a chain of length $3$ and so $Q/\alpha_i\in \LSS(4,p,\sigma_Q)$ for $i=1,\ldots, n$. Therefore $p=7$ (see \cite[Proposition 5.3]{LSS}). Such quandles are not latin, as shown in \cite[Section 5]{LSS}.

\item If $|Q/\nu_Q|=4$, then $[x]_{\nu_Q}$ has size $4p$. Latin quandles are solvable and faithful, so according to Corollary \ref{minimal cong}(iii) the congruence $\nu_Q$ is abelian and $[x]_{\nu_Q}$ is an abelian quandle of size $4p$. In particular $[x]_{\nu_Q}\cong \Q_4\times \aff(\Z_p,f)$ and $[x]_{\nu_Q}$ has $2$ proper congruences. Hence $n=2$ according to Lemma \ref{below abelian}(ii). 

\end{itemize}

We can assume that $\ma_Q=\{\alpha_1,\alpha_2\}$ with $|Q/\alpha_1|=16$ and $|Q/\alpha_2|=4p$. Then $Q/\alpha_1$ is abelian (latin quandles of size $16$ are affine \cite{RIG}), so we have $\gamma_Q\leq \alpha_1$. Then we can conclude that $\gamma_Q=\alpha_1$, since $Q$ is not abelian.  

The size of $Q/\alpha_2$ is $4p$. Since $\alpha_2\wedge \gamma_Q=0_Q$, according to Lemma \ref{central}, $\alpha_2\leq \zeta_Q$ and moreover $Q/\alpha_2$ is not abelian given that $\gamma_Q$ is not contained in $\alpha_2$. Hence $Q/\alpha_2\in \LSS(p,4,\sigma_Q)$. If $\alpha_2<\zeta_Q$ then $\zeta_Q=\nu_Q$ and therefore $\gamma_Q\leq \zeta_Q$. Thus, $Q$ is nilpotent and is therefore directly decomposable according to Lemma \ref{Sub red}, a contradiction. 

Therefore, the congruence lattice of $Q$ is as in \eqref{cong of not SI}.
\end{proof}

   \begin{proposition}\label{group K}
       Let $Q$ be a subdirectly reducible directly indecomposable latin quandle of size $16p$ and $G=\dis(Q)$. Then: 
       \begin{enumerate}
                  \item[(i)] $Z(G)=\dis_{\zeta_Q}=\dis^{\zeta_Q}\cong \Z_2^2$.
           \item[(ii)]  $G/\gamma_2(G)\cong$ \texttt{SmallGroup(32,47)} $\cong \mathcal{Q}_8\times \Z_2^2$. 
           \item[(iii)] $p=1\pmod{3}$ and $G\cong \mathcal{G}_k \times Z(G)$.
              \end{enumerate}

   \end{proposition}

\begin{proof}

(i) Let $G=\dis(Q)$. The blocks of $\zeta_Q$ have size $4$, so they are isomorphic to $\mathcal(Q)_4\cong \aff(\Z_2^2,f)$ and $\gamma_Q\wedge\zeta_Q=0_Q$. According to Lemma \ref{central}, $\dis_{\zeta_Q}=\dis^{\zeta_Q}=Z(G)$ (see Lemma \ref{center for latin}) and has size $4$ (a central group of a transitive group is semiregular). Moreover, $Z(G)$ embeds into a power of $\dis([x]_{\zeta_Q})$ and so $Z(G)\cong \Z_2^2$. 

(ii) Since $Q/\zeta_Q\in \LSS(p,4,\sigma_Q)$ then $|\dis(Q/\zeta_Q)|=8p^2$ (see \cite[Proposition 5.4]{LSS}) and therefore $$|G|=|\dis(Q/\zeta_Q)||\dis^{\zeta_Q}|= 8p^2\cdot 4=|G/\gamma_1(G)||\gamma_1(G)|=16|\gamma_1(G)|.$$ So we have that $|G|=32p^2$ and $\gamma_1(G)$ has size $2p^2$. According to Lemma \ref{dis_gamma} then $\dis_{\gamma_Q}=\gamma_2(G)\cong \Z_p^2$ as it is a non-cyclic elementary abelian $p$-group ($\gamma_Q$ is a minimal abelian non-central congruence).

Let $Q=\mathcal{Q}(G,G_x,f)$ where $f=\widehat{L_x}$ for some $x\in Q$. According to \cite[Lemma 1.10]{Involutory} the principal quandle $Q'=\mathcal{Q}(G/\gamma_2(G),f_{\gamma_2(G)})$ is a connected cover of $Q/\gamma_Q$ of size $32$ and $K=G/\gamma_2(G)\cong \dis(Q')$. In particular $\lambda_{Q'}=\gamma_{Q'}$ is a minimal congruence with blocks of size $2$, $Q'/\lambda_{Q'}\cong Q/\gamma_Q\cong \aff(\Z_2^4,f_{\gamma_1(G)})$ and $\con(Q'/\gamma_{Q'})$ is a chain of length $3$. The quandles of size $32$ are classified in \cite{RIG} and by looking of the size of $\con(Q/\gamma_Q)$ we have $K\cong {\tt SmallGroup(32,i)}$ for $i=47,49$. Assume that $K=G/\gamma_2(G)\cong{\tt SmallGroup(32,49)}$. Let $x=yk\in Z(G)$ for $y\in \Z_p^2$ and $k\in K$. Then $k\in Z(K)$ and so $k\in \gamma_1(K)\leq \gamma_1(G)$. Thus $x\in Z(G)\cap \gamma_1(G)=1$, contradiction. Hence $K\cong {\tt SmallGroup(32,47)}$.




(iii) Note that $G\cong \gamma_2(G)\rtimes_\rho K$. Let $xk\in Z(G)$ with $x\in \Z_p^2$ and $k\in K$. Then $k\in Z(K)\cap \ker{\rho}$ and $x\in Fix(\rho)=\setof{y\in \Z_p^2}{\rho_k(y)=y \text{ for every }k\in K}$. Then $|xk|=|x||k|=2$, therefore $x=1$. Then $Z(G)=\setof{(1,k)}{k\in Z(K)\cap \ker{\rho}}$. Moreover, $K\cong \mathcal{Q}_8\times \Z_2^2$ and $Z(K)=Z(\mathcal{Q}_8)\times \Z_2^2$. Since $\gamma_1(G)\cap Z(G)=1$ we have that $$Z(G)\cap ((\Z_p^2\rtimes \mathcal{Q}_8)\times 1) \leq Z(G)\cap \left(1\times Z(\mathcal{Q}_8)\times 1\right)\leq Z(G)\cap \gamma_1(K)\times 1\leq Z(G)\cap \gamma_1(G)=1.$$ So $G=\left(\Z_p^2\rtimes_\rho \mathcal{Q}_8\times 1\right)\times Z(G)$. Note that $Q/\zeta_Q$ has size $4p$ and is subdirectly reducible, so $p=1\pmod{3}$. Hence $G/Z(G)=G/\dis^\zeta_Q\cong \dis(Q/\zeta_Q)\cong \mathcal{G}_k$.
\end{proof}



\begin{theorem}\label{sr quandles}
Let $p=1\pmod{3}$ and $1\neq k\in \Z_p$ such that $k^3=1\pmod{3}$. The representatives of isomorphism classes of subdirectly reducible directly indecomposable latin quandles of size $16p$ are of the form $\mathfrak{Q}(p,j)=\mathcal{Q}(\mathcal{G}_k\times \Z_2^2,Fix(f(j,a)),f(j,a))$ where
   \begin{align}\label{aut of G_k x Klein_0}
        f(j,a)=\begin{cases}
          x\mapsto y a,\\
                  y\mapsto xy ,\\
                  z\mapsto z\\   f|_{\gamma_2(G)}=\begin{bmatrix}
                      -k(1+k^j) & -(1+k^j)\\
                      0 & -k^2(1+k^j)
                  \end{bmatrix},
                  \\
                  f|_{Z(G)}=\begin{bmatrix}
                      0 & 1\\
                      1 & 1
                  \end{bmatrix},
        \end{cases}
    \end{align}
   for $j=1,2$ and $a=(1,0)$.
\end{theorem}

\begin{proof}
Let $Q$ be a subdirectly reducible latin quandles of size $16p$ and $G=\mathcal{G}_k\times \Z_2^2$. According to Proposition \ref{group K} then $\dis(Q)\cong G$ and so $Q\cong \mathcal{Q}(G,Fix(f),f)$ where $f\in \aut{G}$. Since the factor $Q/\zeta_Q\cong Q(p,i)\cong \mathcal{Q}(G/Z(G),Fix(f_{Z(G)}), f_{Z(G)})$ for some $i=1,2$, then $f_{Z(G)}\in \mathcal{F}=\setof{g\in \aut{\mathcal{G}_k}}{|g|=3,\, |Fix(g)|=2p}$, see \cite{LSS}.
Moreover, the isomorphism classes of these quandles are in one-to-one correspondence with the conjugacy classes of automorphisms of $G$. According to Lemma \ref{iso classes of sr} a set of representatives of conjugacy classes of such automorphisms are as in \eqref{aut of G_k x Klein_0} with $a\in \{(0,0),(1,0)\}$. Note that if $a=(0,0)$ the associated quandle decomposes as the direct product of $Q(p,j)$ and $\mathcal{Q}_4$.
\end{proof}

\subsection{Subdirectly irreducible case}

In this section, we will study subdirectly irreducible latin quandles of size $16p$. We start considering quandles for which the congruence lattice is not a chain.

\begin{proposition}\label{SI 16p case 4p}
Let $Q$ be a latin subdirectly irreducible quandle of size $16p$. If $\con(Q)$ is not a chain we have:
\begin{itemize}
    \item[(i)] $\ma_Q=\{\gamma_Q\}$ and $Z(\dis(Q))=1$.
    \item[(ii)] If $|Q/\gamma_Q|=4p$, then $p\in \{3,5,7\}$ and 
    
     $$
    |\dis(Q)| \text{ divides } \begin{cases} 2^{6}p \text{ if } p\neq 3,\\
    2^{6}\cdot 3^2  \text{ if } p= 3.
    \end{cases}$$
    
    \item[(iii)] If $p>7$ then $Q/\gamma_Q\cong \mathcal{Q}_4^2$. 
\end{itemize}
\end{proposition}


\begin{proof}
Let $G=\dis(Q)$ and $\alpha=\gamma_Q$. The quandle $Q$ is $2$-step solvable, so according to \cite[Corollary 3.14]{Involutory} the prime divisors of $|G|$ are $2$ and $p$.

(i) Since $\con(Q)$ is not a chain, we can assume that $\con(Q)$ is as in Proposition \ref{lattices}(i), and then $|\Ma_Q|\geq 2$. By Lemma \ref{mu and gamma} we see that $\mu_Q=\alpha$ is the unique minimal congruence of $Q$. The quandle $Q$ is not nilpotent, so according to Lemma \ref{dis_gamma}(ii) $Z(G)=1$.

(ii) Assume that $Q/\alpha$ has size $4p$. Then the blocks of $\gamma_Q$ have size $4$, and so the blocks are isomorphic to $\mathcal{Q}_4$ (the unique latin quandle of size $4$). The quandle $Q/\alpha\cong \mathcal{Q}_4\times \aff(\Z_p,k)$ is generated by $(a,0),(0,b)$ for every $0\neq a\in \Z_2^2$ and $0\neq b\neq \Z_p$. According to Corollary \ref{embedding as quandle} we have that $\dis_{\alpha}$ embeds into $\Z_2^{2+k}$ for $0\leq k\leq 2$ and that the group $\dis^{\alpha}=\gamma_1(G)$ embeds into $(\aut{[x]_{\alpha}})^2\cong (\Z_2^2\rtimes \Z_3)^2$ which $3$-Sylow subgroups are isomorphic to $\Z_3^3$. Furthermore, $G/\gamma_1(G)=\dis(Q/\alpha)\cong \Z_2^2\times \Z_p$. We have to discuss two cases.

%
%

\begin{itemize}
    \item Let $p>3$. Then $\gamma_1(G)$ embeds into $(\Z_2^2)^2$ and $|G|=|G/\gamma_1(G)||\gamma_1(G)|=|Q/\alpha||\gamma_1(G)|=4p\cdot 2^t $ for some $2\leq t\leq 4$. Let $x\in G$ be an element of order $p$ and let $H=\langle \gamma_1(G),x\rangle$. Note that $H$ is normal in $G$ and therefore contains the $p$-Sylow subgroups of $G$. Moreover, $H$ is a characteristic subgroup, so in particular $H$ is normal in $\lmlt(Q)$. If $H$ is abelian, then its Sylow subgroups provide trivially intersecting congruences, according to Lemma \ref{Sylow}. So, $H\cong \gamma_1(G)\rtimes_\rho \Z_p$, and accordingly $p$ divides $GL_{2+k}(2)$ where $0\leq k\leq 2$. Thus, $p\in \{3,5,7\}$.

\item Let $p=3$. According to Lemma \ref{N cyclic} the $3$-Sylow of $\gamma_1(G)$ is cyclic and so isomorphic to $\Z_3^s$ for $s=0,1$, so $|\gamma_1(G)|$ divides $2^{2+k}3^s$ for $0\leq k\leq 2$ and $0\leq s\leq 1$. Thus $|G|$ divides $2^{6}3^2$.
\end{itemize}

(iii) We have that $\mu_Q=\alpha$ and it is the unique minimal congruence of $Q$ (see item (i)). The factor $Q/\alpha$ is not simple. Quandles of size $4$ and $p$ are simple, so according to Lemma \ref{meet of max is min}(i) $Q/\alpha$ has size $4p$ or $16$. According to item (ii), then $Q/\alpha$ have size $16$ and is subdirectly reducible. So $Q/\alpha$ is the unique subdirectly reducible latin quandle of size $16$, that is, $\mathcal{Q}_4^2$ (latin quandles of size $16$ are classified in \cite{RIG}).
%
\end{proof}

According to Proposition \ref{SI 16p case 4p} if $p\leq 7$ the size of the displacement group is bounded, so we can study such quandles computationally in the next section. Let us proceed with the case $p>7$.

In the following proof, we use the fact that if $Q=\aff(A,f)$ is a finite connected quandle, then $\aut{Q}\cong A\rtimes C_{\aut{A}}(f)$, so in particular if $A=\Z_2^2\rtimes \Z_p$ for an odd prime $p$, then $\aut{A}=\aut{\Z_2^2}\times \aut{\Z_p}$, consequently $C_{\aut{A}}(f_1,f_2)=C_{\aut{\Z_2^2}}(f_1)\times C_{\aut{\Z_p}}(f_2)$.

\begin{lemma}\label{lemma 4 chain}
Let $Q$ be a latin quandle of size $16p$ such that $\con(Q)$ is a chain of length $4$. Then $\ma_Q=\{ \gamma_Q\}$ and $|Q/\gamma_Q|=16$.




\end{lemma}


\begin{proof}


    
    Let $\Ma_Q=\{\alpha\}$ and $\ma_Q=\{\beta\}$. Note that $|Q/\beta|\in \{16,4p\}$ and $\con(Q/\beta)$ is a chain of length $3$. Assume that $|Q/\beta|=4p$. In this case, we can also assume that $Q/\beta$ is not abelian (otherwise $Q/\beta$ would be directly decomposable), that is, $\gamma_Q=\alpha$. Therefore $p=1\pmod{3}$, since non-abelian latin quandles of size $16p$ exist only for such primes. We have to distinguish between three cases.
    \begin{itemize}

    
    \item Assume that $|Q/\alpha|=p$. Then $Q/\beta\in \LSS(4,p,\sigma_Q)$ and so $p=7$ according to \cite[Proposition 5.3]{LSS}. Such quandles are not latin. 
    
        \item Assume that $|Q/\alpha|=4$ and $[x]_\alpha$ is in $\LSS(4,p,\sigma_Q)$. Then, $p=7$ again, and the blocks of $\alpha$ are not latin.


                \item Assume that $|Q/\alpha|=4$ and $[x]_\alpha\cong \mathcal{Q}_4\times \aff(\Z_p,k)$. 
                We can assume that $p\neq 3$, since $p=1\pmod{3}$. Let $N$ be a $p$-Sylow subgroup of $\dis(Q)$. Since $|\dis(Q/\alpha)|=4$, hence $N$ is contained in $\dis^\alpha$ and since $|\dis(Q/\beta)|=8p^2$ we see that $p^2$ divides $|N|$. On the other hand, $\dis^\alpha$ embeds into $\aut{[x]_\alpha}^2\cong [(\Z_2^2\times \Z_p)\rtimes (\Z_3\times \Z_{p-1})]^2$ since $Q/\alpha$ is $2$-generated (again by using Corollary \ref{embedding as quandle}). The group $\aut{[x]_\alpha}^2$ has a unique normal $p$-Sylow subgroup of size $p^2$ and therefore $|N|=p^2$ and $N$ is normal in $\dis^\alpha$. Then $N$ is characteristic in $\dis^\alpha$ and is therefore normal in $\lmlt(Q)$ accordingly. Hence $\delta=\mathcal{O}_N$ is a congruence of $Q$ with blocks a power of $p$. Hence $\delta=\beta$ and so $|Q/\beta|=16$, contradiction.  
        \end{itemize}
    Therefore, $|Q/\beta|=16$ and so $Q/\beta$ is abelian (as latin quandles of size $16$ are affine). Then $\gamma_Q\leq \beta$ and since $Q$ is not abelian, then $\beta=\gamma_Q$. 
    %
%
%
%
%
\end{proof}


We can use the coset representation of connected quandles in order to conclude that the congruence lattice of latin quandles of order $16p$ is a chain of length $3$ but in a few cases for small primes.

\begin{proposition}\label{chain}
    Let $Q$ be a subdirectly irreducible latin quandle of size $16p$. 
    \begin{itemize}
        \item[(i)] If $\con(Q)$ is a chain then it has length $3$.
        \item[(ii)] If $\con(Q)$ is not a chain then $p\leq 7$.
    \end{itemize}
\end{proposition}





\begin{proof}
Let $G=\dis(Q)$ and $\alpha=\gamma_Q$. Assume that $\con(Q)$ is not a chain and $p>7$ or that $\con(Q)$ is a chain of length $4$. According to Lemma \ref{SI 16p case 4p} and Lemma \ref{lemma 4 chain}, $\alpha$ is the unique minimal congruence of $Q$, the factor $Q/\alpha$ is $\mathcal{Q}_4^2$ and then $G/\gamma_1(G)\cong \Z_2^4$.

Moreover, in both cases, $Q$ is not nilpotent, so $\gamma_2(G)=\dis_\alpha$ and $Z(G)=1$, according to Lemma \ref{dis_gamma}. The blocks of $\alpha$ have size $p$, the quandle $Q/\alpha$ is generated by $3$ elements, so the group $\dis_\alpha$ embeds into $\Z_p^3$ according to Corollary \ref{embedding as quandle}(ii). Moreover $|\dis_\alpha|\neq p$ since $\alpha$ is not central (see Lemma \ref{central 2}).

Let $Q=\mathcal{Q}(G,G_x,f)$ where $f=\widehat{L_x}$ for some $x\in Q$. According to \cite[Lemma 1.10]{Involutory} the principal quandle $Q'=\mathcal{Q}(G/\gamma_2(G),f_{\gamma_2(G)})$ is a connected cover of $Q/\gamma_Q$, $K=G/\gamma_2(G)\cong \dis(Q')$ and it has size a power of $2$. In particular the group $K=G/\gamma_2(G)$ is a $2$-step nilpotent group $2$-group and so $G\cong \gamma_2(G)\rtimes_\rho K$. Moreover $K/\gamma_1(K)\cong G/\gamma_1(G)$ is an elementary $2$-group and so by Lemma \ref{derived si elemetary} $\gamma_1(K)$ is an elementary abelian $2$-group. Moreover $\gamma_1(G)=\gamma_2(G)\rtimes_\rho \gamma_1(K)$. The $2$-Sylow of $\dis^{\gamma_Q}=\gamma_1(G)$ is cyclic according to Lemma \ref{N cyclic}, so the $2$-Sylow of $\gamma_1(G)$ is isomorphic to $\Z_2^s$ for $s=0,1$. 

We have to distinguish between two cases.


\begin{itemize}
    \item Assume that $s=0$. Then $\gamma_1(G)=\gamma_2(G)$ and so $G=\Z_p^t\rtimes_\rho \Z_2^4$ for $t\leq 3$. By Lemma \ref{ro faithful on Z_p}, $\rho$ is not faithful since $t\leq 3\leq 4$, leading to a contradiction.

\item Assume that $s=1$. The quandle $Q'$ is a connected quandle of size $32$ with factor $Q'/\gamma_{Q'}$ isomorphic $\mathcal{Q}_4^2$, having $7$ congruences. Then, looking at Table \ref{32}, we have that $Q'$ is {\tt SmallQuandle(32,j)} for $j=2,3$ and accordingly $K=G/\gamma_2(G)=\dis(Q')$ is {\tt SmallGroup(32,i)} for $i=47,49$. So $G=\Z_p^t\rtimes_\rho K$ for $t=2,3$. Both groups have a subgroup isomorphic to $\Z_2^3$, so we have $t=3$ by Corollary \ref{ro faithful on Z_p}. According to Corollary \ref{center of dis^alpha} then $Z=Z(\gamma_1(G))=Z(\dis^{\alpha})=1$ since $\dis^\alpha$ is not a group of size a power of a prime. On the other Corollary \ref{center = -1} implies that $t$ is even, a contradiction. \qedhere
%
\end{itemize}

%
%
%
\end{proof}

Finally, we deal with quandles with congruence lattice given by a chain of length $3$ and we obtain further constrains of the values of $p$ and on the structure of the displacement group.
\begin{proposition}\label{LSS 16p}
Let $Q$ be a latin quandle of size $16p$ such that $\con(Q)$ is a chain of length $3$. Then: 
\begin{itemize}
\item[(i)]  $Z(\dis(Q))=1$.

\item[(ii)] $\dis(Q)=\Z_2^{4+k}\rtimes_\rho \Z_p$ for $k\leq 4$, and $\rho$ is a faithful action.

\item[(iii)] $|Q/\gamma_Q|=p\in \{3,5,7,17,31,127\}$.




\end{itemize}

\end{proposition}

\begin{proof}
%

Let $G=\dis(Q)$ and $\alpha=\gamma_Q$. Note that $Q$ is solvable, $\alpha\in \ma_Q\cap\Ma_Q$ and so it is abelian. Therefore, both $Q/\alpha$ and $[x]_\alpha$ have size a power of a prime.

(i) Clearly $\alpha$ is the only proper congruence, and so it is minimal. Therefore, $Z(G)=1$ according to Lemma \ref{dis_gamma}.

(ii) We separate the discussion into some cases.

\begin{itemize}
    \item Assume that $|Q/\alpha|$ has size $4$. Then, the blocks of $\alpha$ have size $4p$, which is a contradiction. 


\item If $|Q/\alpha|=16$, then the blocks of $\alpha$ are simple, so $Q\in \LSS(p,16,\sigma_Q)$. But according to Proposition \ref{no LSS}, the class $\LSS(p,16,\sigma_Q)$ is empty.

\item Assume that $|Q/\alpha|$ has size $p$. The blocks of $\alpha$ are affine over $\Z_2^4$. 
The group $G/\gamma_1(G)$ is cyclic and $\gamma_1(G)/\gamma_2(G)$ is generated by the elements $\setof{[x,y]\gamma_2(G)}{x,y\in X}$, where $X$ generates $G/\gamma_1(G)$ (see \cite[Proposition 9.2.5]{Sims}). We conclude that $\gamma_2(G)=\gamma_1(G)$. Since $\dis_{\alpha}$ embeds into $\Z_2^{4+k}$ for $k\leq 4$ (see Corollary \ref{embedding as quandle}), according to Lemma \ref{dis for 3-chains}, we have that $G\cong \Z_2^{4+k}\rtimes_\rho \Z_p$ and $\rho$ is faithful .
\end{itemize}

(iii) The prime $p$ divides $|GL_2(4+k)|$ for $0\leq k\leq 4$, therefore $p\in \{3,5,7,17,31,127\}$. %
%
%
\end{proof}

\subsection{Computational Results}\label{computational}

To classify all subdirectly irreducible latin quandles of size $16p$, we can proceed computationally. Latin quandles are, in particular, connected. Consequently, they admit a minimal coset representation over their displacement group.
    By Lemma~\ref{SI 16p case 4p} and Lemma~\ref{LSS 16p}, such quandles are coset quandles over only finitely many groups. For these reasons, the classification problem reduces to a finite computation. Moreover, once we identify the suitable groups, we can easily obtain isomorphism classes using Theorem \ref{iso theo}.
    
In this section, we describe the algorithm we performed in GAP in order to deal with the cases described above.

    Let $Q$ be a latin quandle of size $16p$. Then $Q$ has a minimal coset representation as $\mathcal{Q}(G,Fix(f),f)$ where $G\cong \dis(Q)$.

\begin{enumerate}
    \item 
    If the congruence lattice of $Q$ not a chain, then, by Lemma~\ref{SI 16p case 4p}, $Q$ is a coset quandle over a centerless group whose order divides $2^6 p$ for $p \neq 3$, or $2^6 \cdot 9$ when $p = 3$ (note that $16p$ must divide $|\dis(Q)|$ in each case since $|\dis(Q)|=|Q||\dis(Q)_x|$ for every $x\in Q$). 
    If the congruence lattice of $Q$ is a chain, then by Lemma~\ref{LSS 16p}, $Q$ is a coset quandles over a group of the form $G_{p,k} = \mathbb{Z}_2^{4+k} \rtimes_\rho \mathbb{Z}_p$, where $0 \leq k \leq 4$ and $p \in \{3, 5, 7, 17, 31, 127\}$, with $Z(G_{p,k}) = 1$. This leads to a similar investigation of potential groups of the correct order, satisfying the appropriate structural conditions.

    \item The second step is to construct all the coset quandles of the form $\mathcal{Q}(G, \fix(f), f)$ over the groups  identified above with the right properties.
    
    \item We proceed by filtering by isomorphism classes using Theorem \ref{iso theo}. 
    \end{enumerate}
    
   The algorithm computes that there are only two such quandles, namely $Q_p=\mathcal{Q}(G_p,Fix(f_p),f_p)$ for $p=3,5$ where the presentations of the groups used for the coset representation are
   \begin{align*}
       G_3 &= \texttt{SmallGroup(48,50)} = \langle a, b, c \ \colon\ 
        a^{3} ,\,
        c^{2} ,\,
        b^{2} ,\,
        (c b)^{2} ,\,
        (b a)^{3} ,\,
        (c a)^{3} ,\,
        (a^{-1} c a b)^{2} ,\,
        (c a^{-1} b a)^{2} 
        \rangle \cong \Z_2^4\rtimes \Z_3,\\ G_5 &= \texttt{SmallGroup(80,49)}=\langle a, b \ \colon \ 
        a^{5} ,\,
        b^{2} ,\,
        (b a^{-1} b a)^{2} ,\, 
        (b a^{-1})^{5} ,\, 
        (b a^{-2} b a^{2})^{2} 
        \rangle\cong \Z_2^4\rtimes \Z_5,
        \end{align*}
    and the automorphisms are defined as
\[
    \begin{aligned}
    f_3 = \begin{cases}
    a \mapsto a^2ba^{-1}ba \\
    b \mapsto ba^{-1}ca \\
    c \mapsto bc
    \end{cases}
    \end{aligned}
    \hspace{2cm}
    \begin{aligned}
    f_5 = \begin{cases}
    a \mapsto a^{2} (a^{2} b)^{2} a^{-3} b a \\
    b \mapsto b a^{2} b a^{-2}
    \end{cases}
    \end{aligned}
    \]  
   Both of these quandles are subdirectly irreducible, and their congruence lattice is a 3-element chain.

\section{Appendix}  \label{sec:appendix}

\subsection{Groups theoretical results}

In this subsection we collect some results we are going to use to describe the displacement groups of latin quandles of size $16p$. In particular, we are going to prove some results on actions of $2$-groups on elementary abelian $p$-groups.


\begin{lemma}\label{derived si elemetary}
    Let $K$ be a $p$-group with nilpotency class $2$. If $K/\gamma_1(K)$ is elementary $p$-abelian group then $\gamma_1(K)$ is elementary $p$-abelian group.
\end{lemma}

\begin{proof}
    Let $x\in G$. The map $[x,-]:y\longrightarrow [x,y]$ is a group homomorphism with values in $\gamma_1(K)$. Moreover, the map factor through the center of $K$. Thus, we have a group homomorphism from $K/Z(K)$ to $\gamma_1(K)$. Since $\gamma_1(K)\leq Z(K)$ then we also have a group homomorphism from $K/\gamma_1(K)$ onto $K/Z(K)$. Hence $K/Z(K)$ is an elementary $p$-abelian group. Therefore, every element of the form $[x,y]$ has order $p$ and so $\gamma_1(K)$ is an elementary $p$-abelian group.
\end{proof}

\begin{lemma}\label{no center}
    Let $H$ be an abelian group, $K$ be a nilpotent group, and $G=H\rtimes_{\rho} K$. If $Z(G)=1$ then $\ker{\rho}=1$.
\end{lemma}

\begin{proof}



We have $\ker{\rho}\cap Z(K)\leq Z(G)=1$. Then $\ker{\rho}=1$ since $K$ is nilpotent.
\end{proof}

Let us show that matrices of order $2$ in $GL_2(p)$ for $p>2$ are diagonalizable.

\begin{lemma}\label{diag nxn}
    Let $A\in GL_n(p)$ of order $k$ and $p$ be an odd prime. Then: 
    
    \begin{itemize}
        \item[(i)] the $A$-irreducible representations has dimension $\leq k$.

 \item[(ii)]    If $k=2$ then $A$ is diagonalizable.
    \end{itemize}
\end{lemma}
\begin{proof}

(i)    Let $v\in \Z_p^n$. Then $\setof{f^i(v)}{i\in \mathbb{Z}}=\{v, f(v),\ldots,f^{k-1}(v)\}$ generates an $A$-invariant subspace of dimension less than or equal to $k$. So, if $V$ is irreducible, then $V$ has dimension less than or equal to $k$.

(ii) The order of $A$ and of $\Z_p^n$ are relatively prime, so according to Maschke theorem, the space $V=\Z_p^n$ decomposes as $V=\oplus_i V_i$ where $V_i$ are irreducible representations of $A$. According to (i), the irreducible representations of $A$ have dimension $1$ or $2$. Let $V_j$ be of dimension $2$ and $B=A|_{V_j}$. According to \cite[Lemma 7.1]{LSS} $B$ is diagonalizable. Thus, also $A$ is diagonalizable.
%
\end{proof}

\begin{corollary} \label{ro faithful on Z_p}
    Let $p$ be an odd prime, $G=\Z_p^n\rtimes_\rho \Z_2^m$. If $\rho$ is faithful then $m\leq n$.
\end{corollary}

\begin{proof}
    According to Lemma \ref{diag nxn} the image of $H$ under $\rho$ is a subgroup of the diagonal matrices of order $2$ (see Lemma \ref{diag nxn}(ii)), which has size $2^n$. Thus, $m\leq n$. 
\end{proof}

\begin{corollary}\label{center = -1}
Let $p$ be an odd prime, $K$ be a $2$-group of nilpotency class $2$ and $G=\Z_p^n\rtimes_\rho K$ where $\rho$ is faithful. If $\gamma_1(K)=\langle z\rangle \cong \Z_2$ and $Z(\Z_p^n\rtimes \gamma_1(K))= 1$ then $\rho_z=-Id_n$ and $n$ is even.    
\end{corollary}

\begin{proof}
The order of $z$ is $2$ and so $\rho_z\neq 1$ and it is diagonalizable with eigenvalues $\lambda\in \{1,-1\}$ according to Lemma \ref{diag nxn}(ii). If $V$ is the eigenspace relative to $\lambda=1$ then $V\leq Z(\Z_p^n\rtimes_\rho \gamma_1(K))=1$. Therefore, $\rho_z=-Id_n$. The map $\det:K\longrightarrow \Z_p^*$ mapping $x\in K$ to $\det(\rho_x)$ is a group homomorphisms into an abelian group. Then $\gamma_1(K)\leq \ker{\det}$ and so $\det(\rho_z)=(-1)^n=1$. Therefore, $n$ is even.
\end{proof}

Let us conclude the section by showing some applications of the previous results to some concrete semidirect product, where the normal summand is an elementary abelian $p$-group and the other summand is a $2$-group.

Let $K_{50}$ be extraspecial $2$-group labeled by {\tt SmallGroup(32,50)} in the {\tt SmallGroup} library of GAP. A presentation of $K_{50}$ is the following:
\begin{align*}
    K_{50}=\langle a,b,c,d\,|\,a^2,\,b^4,\,c^4 d^2,\, cbcb^{-1},\,db^{-1}db,\,c^{-1}aca,\,dc^{-1}dc,\,dadb^2a\rangle.
\end{align*}




\begin{proposition}\label{no rep}
    Let $p$ be an odd prime, $G=\Z_p^2\rtimes_\rho K_{50}$ where $\rho$ is a faithful action. Then $Z(\Z_p^2\rtimes_\rho Z(K_{50}))\neq 1$.
\end{proposition}

    \begin{proof}
 Note that $Z(K_{50})=\gamma_1(K_{50})=\langle b^2\rangle$ and that it has order $2$. Assume that $Z(\Z_p^n\rtimes Z(K_{50}))= 1$ then $\rho_{b^2}=-Id_2$ according to Corollary \ref{center = -1}. Since $|a|=2$, $\rho_a$ is diagonalizable (see Lemma \ref{diag nxn}) with eigenvalues equal to $\pm 1$ and $\rho_a\neq \pm Id_2$. In particular, we can assume that 
$$\rho_a=\begin{bmatrix}
1 & 0\\
0 & -1
\end{bmatrix}.$$ 
Since we have that $[a,c]=1$, $c^2=b^2\in Z(K_{50})$ and $c\notin Z(G)$ we have that $[\rho_a,\rho_c]=1$ and $\rho_a^2=\rho_b^2=\rho_z=-Id$. So necessarily  $$\rho_c=\begin{bmatrix}
    x &0 \\
    0 & -x
\end{bmatrix}$$
for $x^2=-1$. Since $[c,d]=1$ we have $[\rho_c,\rho_d]=1$ and so $\rho_d$ must be a diagonal matrix. Therefore, $[\rho_d,\rho_a]=\rho_{[a,d]}=1$, but $[a,d]\neq 1$, which is a contradiction. 
\end{proof}

\subsection{The Group \texorpdfstring{$\mathcal{G}_k\times \Z_2^2$}{GkxZ2xZ2}}\label{grup for sr}

Let $p=1\pmod{3}$ be a prime number and let $k\in \Z_p$ such that $k^3=1$ and $k\neq 1$. Let us consider the following presentation of the group $\mathcal{Q}_8$:
$$\mathcal{Q}_8=\langle x,y,z\,|\, x^2,\, y^2,\, [x,y]z^{-1},\, z^2,\,[z,y],\,[z,x]\rangle.$$

As defined in \cite[Section 7]{LSS}, we have the group $\mathcal{G}_k=\mathbb{Z}_p^2\rtimes_{\rho} \mathcal{Q}_8$, where the action $\rho$ is defined by setting
\begin{displaymath}
\rho_x=\begin{bmatrix}
0 & -1\\
1 & 0
\end{bmatrix},\quad
\rho_{y}=\begin{bmatrix}
k^2 & k\\
k & -k^2
\end{bmatrix},\quad \rho_z=\begin{bmatrix}
-1 & 0\\
0 & -1
\end{bmatrix}.
\end{displaymath}

Note that $Z(\mathcal{G}_k)=1$, so the center of the group $\mathcal{G}_k\times A$ where $A$ is an abelian group is the subgroup $1\times A$ and $\gamma_n(\mathcal{G}_k\times A)=\gamma_n(\mathcal{G}_k)\times 1$ for every $n\geq 1$.
\begin{lemma}
    The automorphisms of the group $G=\mathcal{G}_k\times \Z_2^2$ are of the form
    \begin{align}\label{aut of G_k x Klein}
        f=\begin{cases}
          x\mapsto \widetilde{f}(x)z_1,\\
                  y\mapsto \widetilde{f}(y)z_2,\\
                  z\mapsto \widetilde{f}(z)\\   f|_{\gamma_2(G)}=\widetilde{f}|_{\gamma_2(\mathcal{G}_k)},\\
                  f|_{Z(G)}=F,
        \end{cases}
    \end{align}
    where $\widetilde{f}\in \aut{\mathcal{G}_k}$, $z_1,z_2\in Z(G)$ and $F\in \aut{Z(G)}$. In particular $|\aut{G}|= 2^8 \cdot 3^2\cdot p^2(p-1)$.
\end{lemma}

\begin{proof}
  Let $f\in \aut{G}$ then clearly $f$ is as in \eqref{aut of G_k x Klein} since $Z(G)$, $\gamma_1(G)$ and $\gamma_2(G)$ are characteristic subgroups. On the other hand, let $f$ be defined as in \eqref{aut of G_k x Klein}. We must check that all relations in the presentation of $G$ are preserved by $f$. 
  \begin{align*}
      [f(x),f(y)]&=[\widetilde{f}(x)z_1,\widetilde{f}(x)z_2]=[\widetilde{f}(x),\widetilde{f}(y)]=\widetilde{f}([x,y])=\widetilde{f}(z)=f(z)=f([x,y]).\\
      f(x)^4&=\widetilde{f}(x)^4 z_1^4=\widetilde{f}(x)^4=1=f(x^4)\\
      f(y)^4&=\widetilde{f}(y)^4 z_2^4=\widetilde{f}(y)^4=1=f(y^4)\\
      f(v^p)&=f(v)^p=f(z^2)=f(z)^2=1, \\
      f(x) f(v) f(x)^{-1}&=\widetilde{f}(x)^{-1}z_1\widetilde{f}(v)z_1^{-1}\widetilde{f}(x)^{-1}=\widetilde{f}(xvx^{-1})=f(xvx^{-1}),\\
      f(y) f(v) f(y)^{-1}&=\widetilde{f}(y)^{-1}z_2\widetilde{f}(v)z_2^{-1}\widetilde{f}(y)^{-1}=\widetilde{f}(yvy^{-1})=f(yvy^{-1}),\\
    [f(z),f(t)]&=1=f([x,t]),
  \end{align*}
for every $v\in \gamma_2(G)$, $z\in Z(G)$ and $t\in G$. Therefore, $f\in \aut{G}$. The second statement is straightforward, using that $|\aut{\mathcal{G}_k}|=24p^2(p-1)$ (see Proposition 7.5 of \cite{LSS}) and $|\aut{\Z_2^2}|=6$.
\end{proof}

Let us define $\mathcal{F}=\setof{h\in \aut{\mathcal{G}_k}}{|Fix(h)|=2p,\, |h|=3}$ and $\mathcal{H}=\setof{h\in \aut{G}}{h_{Z(G)}\in \mathcal{F},\, |h|_{Z(G)}|=3}$. According to \cite[Lemma 7.7]{LSS} $|\mathcal{F}|=16p$ and consequently, $|\mathcal{H}|=16p\cdot 16\cdot 2=2^9p$. The representatives of conjugacy classes of the elements in $\mathcal{F}$ are 
\begin{align} 
        f(j)=\begin{cases}
          x\mapsto y,\\
                  y\mapsto xy,\\
                  z\mapsto z\\   f|_{\gamma_2(\mathcal{G}_k)}=F_j=\begin{bmatrix}
                      -k(1+k^j) & -(1+k^j)\\
                      0 & -k^2(1+k^j)
                  \end{bmatrix},
       \end{cases}
    \end{align}
for $j=1,2$, see the proof of \cite[Theorem 5.5]{LSS}. Let us also define
\begin{align*}
    M=\begin{bmatrix}
        0& 1\\
        1 & 1
    \end{bmatrix}\in GL_2(2).
\end{align*}
    
\begin{lemma}\label{iso classes of sr}
   The representatives of the conjugacy classes of elements of $\mathcal{H}$ are of the form
   \begin{align*}
        f(j,a)=\begin{cases}
          x\mapsto ya,\\
                  y\mapsto xy ,\\
                  z\mapsto z\\   f|_{\gamma_2(G)}=F_j,\\
                  f|_{Z(G)}=M,
        \end{cases}
    \end{align*}
   for $j=1,2$ and $a\in \{(0,0),(1,0)\}$.
\end{lemma}

\begin{proof}
Note that $f(j,a)_{Z(G)}=f(j)$ as defined in \eqref{fj} for $j=1,2$. Clearly the representatives listed in the statement are not conjugate. Indeed if $f(l,a)$ and $f(k,a')$ are conjugate then $f(l,a)_{Z(G)}=f(l)$ and $f(k,a')_{Z(G)}=f(k)$ are conjugate in $\aut{\mathcal{G}_k}$ and so $l=k$. Moreover $|f(l,(0,0))|=3$ and $|f(l,(1,0))|=6$, so they are not conjugate.

Let us compute the centralizer of $f(j,a)$ in $\aut{G}$. If $h\in C(f(j,a))$ then $h|_{Z(G)}$ centralizes $M$ and so $h|_{Z(G)}\in \langle M\rangle$. Moreover $h_{Z(G)}$ centralizes $f(j)$ in $\aut{\mathcal{G}_k}$ and therefore we have that $h$ is one of the following:
\begin{align*}\label{aut of G_k x Klein}
        h_1&=\begin{cases}
          x\mapsto xv_1 z_1,\\
                  y\mapsto y v_2 z_2',\\
                  z\mapsto z w\\   h_1|_{\gamma_2(G)}=\begin{bmatrix}
                      x & 0\\
                      0 & x
                  \end{bmatrix},\\
                  h_2|_{Z(G)}=M^s
       \end{cases},\qquad
       h_2&=\begin{cases}
          x\mapsto xy v_1 z_1',\\
                  y\mapsto x v_2 z_2',\\
                  z\mapsto z w\\   h_2|_{\gamma_2(G)}=\begin{bmatrix}
                      x & -xk\\
                      0 & xk^2
                  \end{bmatrix},\\
                  h_2|_{Z(G)}=M^s
       \end{cases},\qquad
        h_3&=\begin{cases}
          x\mapsto y v_1 z_1',\\
                  y\mapsto xy v_2 z_2',\\
                  z\mapsto z w\\   h_3|_{\gamma_2(G)}=\begin{bmatrix}
                      x & xk\\
                      0 & xk^2
                  \end{bmatrix},\\
                  h_2|_{Z(G)}=M^s
       \end{cases}
    \end{align*}
    for $x\neq 0$, $v_1,v_2,w\in \Z_p^2$, $z_1',z_2'\in \Z_2^2$ and $s=0,1,2$. It is easy to check that if $a=(0,0)$ then $s$ and $z_1'$ can be chosen freely and $z_2'$ is uniquely determined in all three cases. Therefore, $|C(f(j,(0,0),(0,0)))|=|C(f_j)|\cdot 3\cdot 4$. If $z_1=(1,0)$ then in all three cases $s$ and $z_2'$ are uniquely determined and $z_1'$ is free. Therefore $|C(f(j,(1,0),(0,0)))|=|C(f_j)|\cdot 4$. The set $\mathcal{H}$ is invariant under the action by $\aut(\mathcal{G}_k \times \Z_2^2)$ and so the conjugacy class of $f(j,a)$ lies in $\mathcal{H}$. So, given that $|C(f(j,a)_{Z(G)})|=3p(p-1)$ (see \cite[Lemma 7.8]{LSS}) we have that
    \begin{align*}
    2\left(\frac{|\aut{G}|}{|C(f(j,(0,0),(0,0)))|}+\frac{|\aut{G}|}{|C(f(j,(1,0),(0,0)))|}\right)=2( 2^6p+3\cdot 2^6 p)=2^9p=|\mathcal{H}|.    
    \end{align*}
    Therefore the automorphisms in the statement are a set of representatives of conjugacy classes of elements of $\mathcal{H   }$.
    \end{proof}


%
 

 
 %

\section*{Acknowledgments}
\noindent The author Filippo Spaggiari was supported by GAUK (301-10/252012, project number 166122) and the GAČR grant no. 22-19073S.

\bibliographystyle{amsalpha}
\bibliography{references} 

\end{document}